\documentclass[11pt,a4paper]{amsart}

\usepackage{amssymb,amsmath,amsthm,mathrsfs,enumerate}
\usepackage{color, todonotes, comment}
\usepackage{hyperref}

\newtheorem{defin}{Definition}[section]
\newtheorem{prop}[defin]{Proposition}
\newtheorem{theorem}[defin]{Theorem}
\newtheorem{coroll}[defin]{Corollary}
\newtheorem{lemma}[defin]{Lemma}

\theoremstyle{definition}
\newtheorem{rem}[defin]{Remark}

\usepackage[absolute,showboxes]{textpos} 

\newcommand{\eps}{\varepsilon}
\newcommand{\ud}{\mathrm{d}}

\newcommand{\C}{\mathscr{C}}

\newcommand{\R}{\mathbb{R}}

\TPMargin{5pt}

\newcommand{\copyrightstatement}{
    \begin{textblock}{0.62}(0.19,0.93)    
         \noindent
         \footnotesize
         \copyright  The published version of this paper is available online (June 2019) on \emph{Journal of Differential Equations}.  DOI \url{10.1016/j.jde.2019.06.014}
    \end{textblock}
}

\begin{document}

 \copyrightstatement 

\title[A non-linear  PDE with a distributional coefficient]{A non-linear parabolic PDE with a distributional coefficient and its applications to stochastic analysis}
\author{Elena Issoglio}
\address{School of Mathematics, University of Leeds, LS2 9JT, Leeds, UK}
\email{e.issoglio@leeds.ac.uk}
\date{\today}

\begin{abstract}
We consider a non-linear parabolic partial differential equation (PDE) on $\mathbb R^d$ with a distributional coefficient in the non-linear term. The distribution is an element of a Besov space with negative regularity and the non-linearity is of  quadratic type in the gradient of the unknown. Under suitable conditions on the parameters we prove local existence and uniqueness of a mild solution to the PDE, and investigate properties like continuity with respect to the initial condition and blow-up times. 
We prove a global existence and uniqueness result assuming further properties on the non-linearity. To conclude we consider an application of the PDE to stochastic analysis, in particular to a class of non-linear  backward stochastic differential equations with distributional drivers.
\end{abstract}


\maketitle

\section{Introduction}\label{sc:setting of the problem}
In this paper we consider the following non-linear parabolic equation 
\begin{equation}\label{eq: PDE non lin iniziale}
\left\{\begin{array}{lcr}
\frac{\partial u(t,x)}{\partial t}=  \Delta u (t,x) + F( \nabla u(t,x))  b (t, x),&\quad&  x\in \mathbb R^d ,t   \in (0,T]\\
u(0,x)=u_0(x), && x\in \mathbb R^d
\end{array}\right.
\end{equation}
where $u:[0,T]\times \mathbb R^d \to \mathbb R$ is the unknown, $b:[0,T]\times \mathbb R^d \to \mathbb R$ is a given (generalised) function  and $u_0: \mathbb R^d \to \mathbb R$ is a suitable initial condition.
 Here the gradient operator $\nabla$ and the Laplacian $\Delta$ refer to the space component.  The term $F: \R^d \to \R$ is a non-linear map of quadratic type whose regularity will be specified below (see Assumption A1). 

In this paper we are interested in the case when the coefficient $b$ is highly singular in the space component, in particular we will consider bounded functions of time taking values in a suitable class of Schwartz  distributions, $b\in L^\infty([0,T]; \mathcal C^{\beta}(\mathbb R^d))$ for some $\beta\in (-1/2 ,0)$. Here $\mathcal C^\beta$ is a Besov space whose exact definition will be recalled later. 

The main motivation for looking at this kind of \emph{rough} equations with singular coefficients comes from Physics. In recent years there has been a great interest in the study of stochastic partial differential equations (SPDEs), fuelled by the success of the theories of regularity structures by Hairer \cite{hairer14} and of paracontrolled distributions by Gubinelli and coauthors \cite{gubinelli04, gubinelli-imkeller-perkowski, gubinelli-perkowski15}. These two theories allowed for the first time to study stochastic PDEs with very singular coefficients (such as  the Kardar-Parisi-Zhang equation, see \cite{hairer13}) which posed long standing problems. Amongst the many papers in the area of stochastic PDEs that build on these ideas, we mention a series of recent ones on quasilinear stochastic PDEs \cite{ballieul_et.al, furlan_gubinelli,  gerencser_hairer, otto_et.al, otto_weber} that may  be of interest  to the reader.

Also in the present paper we  consider a quasilinear PDE, but a deterministic one, where one of the  coefficients is \emph{singular} because it is a distribution. This coefficient however, is \emph{regular} enough to allow for \emph{Young-type} products to be used.  
This approach is the same in spirit as \cite{hinz_et.al, hinz_zahle, issoglio13, issoglio_zahle}, where the authors look for solutions to linear and non-linear parabolic PDEs for which  some of the coefficients are distributions that may arise as realisations of stochastic noises. The aim of these papers, as well as the present work,  is to solve such PDEs with \emph{classical} techniques and in particular without using any special properties of the coefficients that derive from the fact that they are the realization of a stochastic noise -- hence avoiding to use the machinery mentioned above for SPDEs. This of course will result in restrictions on the (ir)regularity of the distributional coefficient $b$ (which would play the role of the space-time noise in the SPDEs context). In the present paper, the non-linearity $F$ is assumed to be continuously differentiable with Lipschitz partial  derivatives, hence allowing for quadratic growth.
To the best of our knowledge this is the first time that existence and uniqueness of mild solutions for \eqref{eq: PDE non lin iniziale} is studied in the literature. It may be worth emphasizing that the key technical difficulty is that the non-linearity involves the gradient of the unknown (as for example in the Burger's equation). This is different to \cite{issoglio_zahle} where the non-linearity involves the solution itself. In both cases, the non-linear term is `multiplied' by the distributional coefficient.

 Our main result is {\em local existence and uniqueness of a mild solution} in $C([0,T]; \mathcal C^{\alpha+1})$, where $\alpha>0$ depends on $\beta$ (see Assumption A2 below). Here local solution means  either a solution with an arbitrary initial condition and a sufficiently small time $T$ (see Theorem \ref{thm: local fixed point for J}) or with an arbitrary time $T$ but a sufficiently small (in norm) initial condition (see Theorem \ref{thm: fixed point for J for small u0}).  Both theorems are proven  with a  fixed point argument and careful a-priori bounds  on the quadratic non-linearity $F$.  We also show continuity of the solution with respect to the initial condition (Proposition \ref{pr: continuity wrt u0}) and we start to investigate blow-up times for the solution (see Proposition \ref{pr: blow up}).  
 
The quadratic growth of the non-linearity $F$ is the main issue that prevents us from finding a global in time solution. Indeed, if we assume that $F$ is Lipschitz with sub-linear growth (see Assumption A4) then we can show {\em existence and uniqueness of a  global mild solution} in $ C^\eps([0,T]; \mathcal C^{\alpha+1})$ for some $\eps>0$ and  for all times $T<\infty$ (see Theorem \ref{thm:global}).

To conclude the paper we illustrate an application of  PDE \eqref{eq: PDE non lin iniziale}  to stochastic analysis, in particular to a class of non-linear backward stochastic differential equations (BSDEs) with singular coefficients. This example falls in the class of quadratic BSDEs and the novelty is  the presence of a distributional coefficient in the so-called \emph{driver} of the BSDE. The study of quadratic BSDEs has been initiated in 2000 by Kobylanski \cite{kobylanski}, while BSDEs with singular terms (mostly linear)  have started gaining attention only recently, see e.g.\ \cite{diehl_friz12, diehl_zhang17, IssoglioJing16, issoglio_russo}. To the best of our knowledge, the only paper that deals with {\em singular} quadratic BSDEs is \cite{eddahbi}, but the singular term is a linear stochastic integral with respect to a rough function, unlike in the present paper where the singularity appears in the quadratic term. 

\vspace{5pt}
The paper is organised as follows: In Section \ref{sc:preliminaries} we recall known results that will be needed later, including the definition of product between distributions and the definition of the function spaces used.  In Section \ref{sc: solving the PDE} we show useful properties of the integral operator appearing in the mild solution and show all necessary a priori bounds and contraction properties. Using those we prove the main result of local existence and uniqueness of a mild solution (Theorems \ref{thm: local fixed point for J} and \ref{thm: fixed point for J for small u0}). We also investigate continuity with respect to initial condition and blow-up of the solution. In Section \ref{sc:global} we study global existence and uniqueness of a mild solution (see Theorem \ref{thm:global}) under more restrictive assumptions on the non-linearity.
 Finally  in Section \ref{sc: applications} we apply these results to stochastic analysis, and give a meaning and solve a class of non-linear BSDEs with distributional coefficients.

For ease of reading we collect here some of the function spaces  used more often in this papers (and point the reader to the precise definition in the section below when needed). We have
\begin{itemize}
\item $C_TX:= C([0,T];X)$, that is the space of $X$-valued continuous functions  defined on $[0,T]$ for any Banach space $X$, see Section \ref{sc:preliminaries}
\item $L^\infty_TX:= L^\infty(0,T;X)$, that is the space of $X$-valued $L^\infty$-functions  defined on $[0,T]$ for any Banach space $X$
\item $\mathcal C^\gamma : = B^\gamma_{\infty, \infty}$, where the Besov spaces $B^\gamma_{p,q}$ are defined in \eqref{eq: Besov spaces alpha p q}
\item $C_T\mathcal C^{\alpha+1}$ is then a particular case (often used below) and this  is the space of continuous functions of time  defined on $[0,T]$ taking values in the Besov space  $\mathcal C^{\alpha+1}$
\item $ C^\eps_T\mathcal C^{\alpha+1}:= C^\eps([0,T];\mathcal C^{\alpha+1} )$ is the space of $\eps$-H\"older continuous functions on $[0,T]$ taking values in the Besov space  $\mathcal C^{\alpha+1}$, see Section \ref{sc:global}
\end{itemize}

\section{Preliminaries}\label{sc:preliminaries}

\subsection{Fractional Sobolev spaces, semigroups and products}\label{ssc: sobolev sp and semigroups}
We start by recalling the definition of Besov spaces $B^\gamma_{p,q}$ on $\mathbb R^d$ for $\gamma\in \mathbb R$ and $1 <p,q\leq \infty$. For more details see for example Triebel \cite[Section 1.1]{triebel10} or Gubinelli \cite[Appendix A.1]{gubinelli-imkeller-perkowski}. Let $\mathcal S'$ be the space of real valued  Schwartz distributions on $\mathbb R^d$. 
We denote by $|\cdot|_d$ the Euclidean norm in $\mathbb R^d$. For $x, y \in \mathbb R^d$ we write $x\cdot y$ to denote the scalar product in $\mathbb R^d$. 
Let us consider a dyadic partition of unity $\{\phi_j, j\geq 0\}$ with the following properties: the zero-th element is such that 
\[
\phi_0(x) = 1 \text{ if } |x|_d\leq 1 \quad \text{ and } \quad \phi_0(x) = 0 \text{ if } |x|_d\geq \frac32
\] 
and the rest satisfies 
\[
\phi_j (x) = \phi_0 (2^{-j}x)- \phi_0(2^{-j+1}x)  \text{ for } x\in \mathbb R^d \text{ and } j\in \mathbb N.
 \]
We define
\begin{equation}\label{eq: Besov spaces alpha p q}
B^\gamma_{p,q}:= \left\{ f\in \mathcal S' \, : \, \|f\|_{B^\gamma_{p,q}}:= \left( \sum_{j=0}^\infty 2^{\gamma j q} \|(\phi_j \hat f )^\vee\|_{L^p}^q \right)^{1/q}<\infty \right\},
\end{equation}
where $\hat \cdot $ and $()^\vee$ denote the Fourier transform and its inverse, respectively. If $q=\infty$ in \eqref{eq: Besov spaces alpha p q} we consider the usual modification of the norm as follows
\begin{equation*}
\|f\|_{B^\gamma_{p,\infty}}:= \sup_{j}  2^{\gamma j } \|(\phi_j \hat f )^\vee\|_{L^p}
\end{equation*}
In the special case where both $p=q=\infty$ in \eqref{eq: Besov spaces alpha p q}, we use a different notation for the Besov space, namely $\mathcal C^\gamma := B^\gamma_{\infty, \infty}$. The norm in this space will be denoted by $\|\cdot\|_\gamma$. Note that the norm depends on the choice of the dyadic partition of unity $\{\phi_j\}$ but the space $B^\gamma_{p,q}$ does not, and all norms defined with a different $\{\phi_j\}$ are equivalent.  In the case when $0<\gamma<1$ we will sometimes use yet another equivalent norm in $\mathcal C^\gamma$ which is given by 
\begin{equation}\label{eq: equivalent norm C alpha}
\sup_{x\in\mathbb R^d } \left( |f(x)| +  \sup_{0<|h|_d\leq 1}  \frac{|f(x+ h)-f(x) |}{|h|_d^{\gamma}}\right),
\end{equation}
see \cite[equation (1.22) with $m=1$]{triebel10}. 
Note moreover that for a non-integer  $\gamma>0$,  the space $\mathcal C^\gamma$ is the usual space of functions  differentiable $m$ times (with $m$ being the highest integer smaller than $\gamma$), with bounded partial derivatives up to order $m$ and whose partial derivatives of order $m$ are ($\gamma-m$)-H\"older continuous (see \cite[page 99]{bahouri}). 
On the other hand, if $\gamma<0$ then the space $\mathcal C^\gamma$ contains distributions. 
Besov spaces are well suited to give a meaning to multiplication between distributions. Indeed using Bony's estimates (see \cite{bony}) one can show that  for $f\in \mathcal C^\gamma$ and $g\in \mathcal C^\delta$ with $\gamma+\delta>0$ and $\delta<0$, then $fg$ exists as an element of $\mathcal C^\delta$ and
\begin{equation}\label{eq: Bony's estimates}
 \| fg  \|_\delta \leq c \| f\|_\gamma  \| g\|_\delta,
\end{equation}
for some constant $c>0$, see \cite[Lemma 2.1]{gubinelli-imkeller-perkowski} for more details and a proof.

For a Banach space $X$, let  $C_T X:= C([0,T]; X)$ denote  the space of $X$-valued continuous functions of time. This is a Banach space endowed with the usual supremum norm
$$\| u\|_{C_T X}:= \sup_{t\in[0,T]} \|u(t)\|_{X}$$
for  $u\in C_T X$.  
 On the same space  $C_T X$  we consider a family of equivalent norms $\|\cdot\|^{(\rho)}_{C_T X}, \rho\geq 1$ given by
\begin{equation}\label{eq: equivalent norm}
\|u\|^{(\rho)}_{C_T X} := \sup_{t\in [0,T]} e^{-\rho t}\|u(t)\|_X
\end{equation}
for  $u\in C_T X$. On the space $L_T^\infty X:= L^{\infty}([0,T]; X)$, where $X$ is a Banach space, we consider the norm  ${\mathrm{ess sup}}_{t\in[0,T]} \|f(t)\|_{X}$ for a function $f:[0,T]\to X $ and we denote it 	by $\| f \|_{L_T^\infty X}$.


It is useful to  rewrite  equation \eqref{eq: PDE non lin iniziale} as the following  abstract Cauchy problem
\begin{equation}\label{eq: PDE non lin Cauchy prb}
\left\{\begin{array}{lcl}
\frac{d u(t)}{d t}=  \Delta u(t) +F (\nabla u(t))  b(t) &\quad&   \text{on }  \mathbb R^d \times(0,T]\\
u(0)=u_0, &&
\end{array}\right.
\end{equation}
where now $u$ denotes a function  of time with values in an infinite dimensional space  that will be specified later. The same notation is applied to the field $b$. We are now ready to introduce explicitly  the notion of solution of \eqref{eq: PDE non lin iniziale} considered in this paper. 
\begin{defin}
We say that $u\in C_T \mathcal C^{\alpha +1}$ is a mild solution of \eqref{eq: PDE non lin iniziale}  or equivalently \eqref{eq: PDE non lin Cauchy prb} if it satisfies the following integral equation
\begin{equation}\label{eq: mild solution}
u(t)= P_t u_0+\int_{0}^t P_{t-s}\left(  F(\nabla u(s))  b(s) \right) \mathrm ds,
\end{equation}
where $\{P_t\}_{t\geq 0} $ is the heat semigroup acting on the product $ F(\nabla u(s))  b(s) $. 
\end{defin}
 The generator of  $\{P_t\}_{t\geq 0} $ is the Laplacian $\Delta$ and the semigroup acts on $\mathcal S'$ but as an operator it can be restricted to $\mathcal C^\gamma$ for any $\gamma$. It is known that the heat semigroup $P_t$ enjoys  useful properties as a mapping on the $\mathcal C^\gamma$-spaces, 
for example the well-known \emph{Schauder's estimates} (see e.g.~\cite[Lemma A.8]{gubinelli-imkeller-perkowski} or \cite[Prop.\ 2.4]{cannizzaro}) recalled in the following.
Let  $\theta\geq0 $ and $\gamma\in \mathbb R$. For any  $g\in \mathcal C^\gamma$ and $t>0$ then $P_t g\in \mathcal  C^{\gamma+2\theta}$ and 
\begin{equation}\label{eq: mapping Pt Besov spaces}
\|P_tg\|_{\gamma+2\theta} \leq c t^{-\theta} \|g\|_{\gamma}
\end{equation}
and
\begin{equation}\label{eq: mapping Pt-I Besov spaces}
\|(P_t-1)g\|_{\gamma-2\theta} \leq c |t|^{\theta} \|g\|_{\gamma}.
\end{equation}

\subsection{Assumptions} We list here the main assumptions that we will use throughout the paper on the non-linear term $F$,  on the parameters  $\alpha,  \beta$ and on the distributional term $b$.
\begin{description}
\item[A1] \textbf{Assumption on non-linear term $F$.}
\emph{Let  $F:\mathbb R^d\to \mathbb R$  be a $\C^1$-function whose partial derivatives $\frac{\partial}{\partial x_i} F$ are Lipschitz with the same constant $L$ for all $i=1, \ldots, d$.}
\end{description} 
Note that from Assumption A1 it follows that there exists a  positive constant $l$ such that 
\[
\left | \frac{\partial  F}{\partial x_i}(x)\right | \leq l (1+|x|_d)
\]
for all $i=1,\ldots, d$. 
The key example we have in mind is the {\em quadratic non-linearity} $F(x)=x^2$ (in dimension $d=1$).

Using $F$ we define an operator $\mathrm F$ as follows: for any element $f\in \mathcal C^\alpha$ for some $\alpha>0$ we define the function $\mathrm F(f)$ on $\mathbb R^d$ by
\begin{equation}\label{eq: operator F}
\mathrm F(f)(\cdot):=F(f(\cdot)).
\end{equation}
\begin{description}
\item[A2] \textbf{Assumption on parameters.}
\emph{We choose $0<\alpha<1$ and $\beta<0$ such that  $\max\{-\alpha, \alpha-1\}<\beta$. In particular this implies $-\frac12<\beta<0$.} 
\item[A3] \textbf{Assumption on $b$}.
\emph{We take $b\in L_T^\infty \mathcal C^\beta$.}
\end{description}

\section{Solving the PDE}\label{sc: solving the PDE}

\subsection{On the non-linear term}\label{ssc: non-linear term}
 In this section we prove a technical result that will be  key to control the non-linear term in equation \eqref{eq: PDE non lin Cauchy prb} when applying a fixed point argument later on.  We state and prove the result for the operator $\mathrm F$ applied to functions $f$ and $g$ with the same regularity as $\nabla u(s)$ will have.

\begin{prop}\label{pr: mapping prop non linear F}
Let $F: \mathbb R^d \to \mathbb R$ be a non-linear function that satisfies Assumption A1.  Then the operator $\mathrm F$ defined in \eqref{eq: operator F} is a  map
 \[
 \mathrm F: \mathcal C^\alpha \to \mathcal C^\alpha
 \]
  for any $\alpha\in(0,1)$. In particular if $\mathbf 0$ denotes the zero-function then  $\|\mathrm F(\mathbf 0)\|_\alpha = |F(0)|$. Moreover for  $f, g: \mathbb R^d \to \mathbb R^d$ elements of  $\mathcal C^\alpha$ component by component then we have
\begin{align}\label{eq: mapping prop non linear F}
\|\mathrm F(f)-\mathrm F( g)\|_{\alpha} & \leq c( 1+ \| f \|_\alpha^2 +\| g \|_\alpha^2 )^{1/2} \|f-g\|_{\alpha}
\end{align}
where the constant $c$  depends on $ L, l$ and $d$.
\end{prop}

\begin{proof}
 For simplicity of notation we will omit the brackets and sometimes write $\mathrm Ff-\mathrm Fg$ instead of $\mathrm F(f)-\mathrm F(g)$ for $f, g \in \mathcal C^\alpha$. We recall that a function is an element of  $\mathcal C^\alpha$ if its norm is bounded. Moreover for $0<\alpha<1$ we can use the equivalent norm \eqref{eq: equivalent norm C alpha}. 
 
 We want to bound
\begin{align}\label{eq: eq for Ff-Fg}
\| \mathrm Ff-\mathrm Fg \|_\alpha :=& \sup_{x\in \mathbb R^d} |Ff(x) - Fg(x)| \nonumber \\
+&  \sup_{0<|y|_d\leq 1} \sup_{x\in \mathbb R^d} \frac{|Ff(x+y) - Fg (x+y) - Ff(x) + Fg(x)|}{|y|_d^\alpha}.
\end{align}
Using the  $ \C^1$ assumption on $F$, we have for   $ a,b \in \R^d$ and   $\theta\in[0,1]$ that	 
\begin{align*}
\frac{\mathrm d}{\mathrm d \theta} F( \theta a +(1-\theta)b ) &= \sum_{i=1}^d \frac \partial {\partial x_i} F ( \theta a +(1-\theta)b ) (  a_i -  b_i ),
\end{align*}
 and so integrating from 0 to 1 in $\mathrm d \theta$ one has
\[
  F(a)- F(b) =  \int_0^1 \nabla F(\theta a+(1-\theta)b) \,\ud \theta \cdot (a-b). 
\]
Furthermore using the linear growth assumption on each component $\frac{\partial}{\partial x_i}F$ of $\nabla F$ and Jensen's inequality  we get
\begin{align}  \label{eq: F lipschitz bound}
\vert F(a)-F(b)\vert  & \leq \nonumber  
  \vert a-b \vert_d  \int_0^1 \left( \sum_{i=1}^d  \left\vert  \frac\partial{\partial x_i} F(\theta a+(1-\theta)b)\right\vert^2 \right)^{1/2} \mathrm d \theta\\
  & \leq   c \vert a-b \vert_d  \int_0^1  \left(  \sum_{i=1}^d  l^2 (1+ |\theta a+(1-\theta)b |_d)^2 \right)^{1/2} \mathrm d \theta \\ \nonumber 
  & \leq   c \vert a-b \vert_d  \int_0^1  \left(  \sum_{i=1}^d  l^2 (1+ \theta^2 |a|^2+(1-\theta)^2 |b |^2_d) \right)^{1/2} \mathrm d \theta \\ \nonumber 
& \leq    c\sqrt d l  \vert a-b \vert_d     (   1+ |a|^2_d+ |b|^2_d)^{1/2} .
\end{align} 
Hence for the first term in \eqref{eq: eq for Ff-Fg} we get
\begin{align*}
\sup_{x\in \mathbb R^d} |Ff(x) - Fg(x)| &\leq c \sup_{x\in \mathbb R^d} |f(x) - g(x)|  (1+ |f(x)|^2_d+ |g(x)|^2_d)^{1/2}\\
&\leq c \|f-g\|_\alpha (1+\|f\|_\alpha^2 + \|g\|_\alpha^2 )^{1/2}.
\end{align*}
Let us now focus on the numerator appearing in the second  term  of \eqref{eq: eq for Ff-Fg}. Inside  the absolute value  we use twice a computation similar to the one used above and add and subtract the same quantity  to  get
\begin{align*}
 &|Ff(x+y) -Ff(x)-Fg(x+y)+Fg(x)| \\ 
= &  \Big | \int_0^1 \nabla F ( \theta f(x+y) + (1-\theta)f(x)) \mathrm d \theta \cdot(f(x+y)-f(x))\\
& -  \int_0^1 \nabla F ( \theta g(x+y) + (1-\theta)g(x)) \mathrm d \theta \cdot (g(x+y)-g(x)) \Big|_d \\ 
\leq &   \int_0^1  \left | \nabla F ( \theta f(x+y) + (1-\theta)f(x)) \right|_d \mathrm d \theta  \\
& \left|f(x+y)-f(x) - g(x+y)+g(x) \right|_d\\
& +   \Big| \int_0^1\left [ \nabla F ( \theta f(x+y) + (1-\theta)f(x)) -  \nabla F ( \theta g(x+y) + (1-\theta)g(x))\right] \mathrm d \theta \\
 &    \cdot(g(x+y)-g(x) ) \Big|
\end{align*}
The first term can be bounded similarly as in \eqref{eq: F lipschitz bound} by
 $$c ( 1+  \|f\|^2_\alpha )^{1/2} |f(x+y)-f(x) - g(x+y)+g(x)|_d .$$
For the second term above, we first observe that since $\frac{\partial}{\partial x_i}  F:\mathbb R^d \to \mathbb R$ is Lipschitz by assumption for all $i$, then $\nabla F: \mathbb R^d \to \mathbb R^d$ is Lipschitz with constant   $L\sqrt d  $. Thus we get the upper bound
\begin{align}\label{eq: second summand non-linear F}
 \nonumber
&|g(x+y)-g(x) |_d \sqrt d L \\ \nonumber
&\int_0^1 \left| 
\theta f(x+y) + (1-\theta)f(x) -\theta g(x+y) - (1-\theta)g(x)\right|_d \mathrm d\theta \\
\leq& c |g(x+y)-g(x)|_d  \| f-g\|_\alpha.
\end{align}
Putting everything together for both terms in \eqref{eq: eq for Ff-Fg} we get the bound
\begin{align*}
 &\| \mathrm Ff-\mathrm Fg \|_\alpha \\ 
 \leq & c  \sup_{0<|y|_d\leq 1} \sup_{x\in\mathbb R^d} \Big [ ( 1+   \|f\|_\alpha^2 )^{1/2} \frac{ |f(x+y)-f(x) - g(x+y)+g(x)|_d}{|y|^\alpha_d}\\
&+ \|f-g\|_\alpha\frac{ |g(x+h)-g(x)|_d}{|y|^\alpha_d } \Big] \\
\leq & c ( 1+   \|f\|_\alpha^2 )^{1/2} \|f-g\|_\alpha + \|f-g\|_\alpha \|g\|_\alpha\\
\leq & c \|f-g\|_\alpha  ( 1+   \|f\|_\alpha^2 + \|g\|^2_\alpha )^{1/2} 
\end{align*}
having used again the  equivalent norm \eqref{eq: equivalent norm C alpha}. This shows \eqref{eq: mapping prop non linear F} and in particular that $\mathrm Ff-\mathrm Fg \in \mathcal C^\alpha$. \\
Let us denote by $k:= F(0)$. Then clearly  $\mathrm F\mathbf 0\equiv k$ and
  \begin{align*}
 \|\mathrm F\mathbf 0 \|_\alpha
 &= \sup_{x\in \mathbb R^d} |(\mathrm F \mathbf 0) (x)| + \sup_{0<|y|_d\leq 1} \sup_{x\in \mathbb R^d} \frac{|(\mathrm F\mathbf 0)(x+y) - (\mathrm F\mathbf 0)(x)|}{|y|_d^\alpha}\\
& = \sup_{x\in \mathbb R^d} |k| + 0\\ 
&= |k|.
 \end{align*}
Finally to show that $\mathrm F$ maps $\mathcal C^\alpha$ into itself it is enough to observe that 
\[
\|\mathrm Ff\|_\alpha \leq  \|\mathrm F f - \mathrm F \mathbf 0\|_\alpha + |k|
\] 
 and then the RHS of the above equation is finite by \eqref{eq: mapping prop non linear F} hence $\mathrm Ff\in \mathcal C^\alpha$ for all $f\in \mathcal C^\alpha$.
\end{proof}

\subsection{Existence and Uniqueness}

Let us denote by $J_t(u)$ the right-hand side of (\ref{eq: mild solution}), more precisely 
\begin{equation}\label{eq: operator J}
J_t(u):=P_t u_0+ I_t(u),
\end{equation}
where the integral operator $I$ is given by 
\begin{equation}\label{eq: operator I}
  I_t(u):= \int_{0}^t P_{t-s} \left( \mathrm F(\nabla u(s)) b(s)\right) \mathrm ds
\end{equation}
and the semigroup $P_{t-s}$ acts on the whole product $\mathrm F(\nabla u(s)) b(s)$.

Using Schauder's estimates it is easy to show that $t\mapsto I_t(u)$ is continuous from $[0,T]$ to $\mathcal C^{\alpha+1}$. We show the result below for a general $f$ in place of $ F(\nabla u(s)) b(s)$. Note that the result might look not sharp  because one normally gains 2 derivatives in parabolic PDEs when using semigroup theory (and possibly some time regularity too). Here we  gain slightly less than 2 derivatives (we go from $\beta$ to $\alpha+1$ and $\alpha+1-\beta<2$) because we need the time singularities $t^{-\theta}$ and $t^{-\frac{\alpha+1-\beta}{2}}$ to be integrable. We will   investigate  the time regularity, that is, H\"older continuity in time of small order, later in Section \ref{sc:global}.

\begin{lemma}\label{lm: continuity of I}
Let $\alpha, \beta$ satisfy Assumption A2. Let $f\in L_T^\infty \mathcal C^{\beta}$. Then $\mathcal I_\cdot (f)\in C_T\mathcal C^{\alpha+1}$, where $\mathcal I_t(f):= \int_0^t P_{t-s}f(s)\mathrm ds$.
\end{lemma}
\begin{proof}
We first observe that for fixed $0\leq s\leq t\leq T$ then $P_{t-s}f(s)\in \mathcal C^{\alpha+1}$ by  \eqref{eq: mapping Pt Besov spaces}. The singularity in time is still integrable if $\alpha $ and $\beta$ satisfy Assumption A2. To show continuity of $\mathcal I$ we take some $\varepsilon>0$ and we bound  $\mathcal I_{t+\varepsilon}(f) - \mathcal I_t(f)$ in the space $\mathcal C^{\alpha+1}$ by
\begin{align*}
&\|\int_0^t P_{t-s}(P_\varepsilon f(s) ) \mathrm ds  + \int_t^{t+\varepsilon} P_{t+\varepsilon-s} f(s)  \mathrm ds  - \int_0^t P_{t-s}f(s) \mathrm ds   \|_{\alpha+1}\\
\leq & \| \int_0^t P_{t-s}(P_\varepsilon f(s) -f(s)) \mathrm ds  \|_{\alpha+1} + \| \int_t^{t+\varepsilon} P_{t+\varepsilon-s} f(s)  \mathrm ds \|_{\alpha+1}.
\end{align*}
Now we use Schauder's estimates \eqref{eq: mapping Pt Besov spaces} and \eqref{eq: mapping Pt-I Besov spaces} with some $\nu>0$ such that $\theta:=\alpha+1-\beta+2\nu<2$ (which always exists by Assumption A2) and we get
\begin{align*}
\|\mathcal I_{t+\varepsilon}(f)&-\mathcal I_{t}(f)\|_{\alpha +1} \\\leq  & c \int_0^t (t-s)^{-\frac \theta 2} \| P_\varepsilon f(s) -f(s)\|_{\beta-2\nu}  \mathrm ds  \\
& + c \int_t^{t+\varepsilon} (t+\varepsilon-s)^{-\frac{\alpha+1-\beta}{2}} \|f(s)\|_{\beta} \mathrm ds\\
 \leq&c \int_0^t (t-s)^{-\frac\theta 2} |\varepsilon|^\nu \| f(s)\|_{\beta}  \mathrm ds  \\
 & + c \int_t^{t+\varepsilon} (t+\varepsilon-s)^{-\frac{\alpha+1-\beta}{2}} \|f(s)\|_{\beta} \mathrm ds\\
\leq &c\|f\|_{L_T^\infty\mathcal C^\beta} \left( |\varepsilon|^\nu \int_0^t (t-s)^{-\frac\theta 2}  \mathrm ds + \int_t^{t+\varepsilon} (t+\varepsilon-s)^{-\frac{\alpha+1-\beta}{2}}  \mathrm ds \right)\\
\leq &c \|f\|_{L_T^\infty\mathcal C^\beta} \left( |\varepsilon|^\nu   t^{-\frac\theta 2+1} +    \varepsilon^{\frac{-\alpha+1+\beta}{2}}   \right),
\end{align*}
and the latter tends to 0 as $\varepsilon\to0$ for all $t\in[0,T]$ because $\nu>0$ and $-\frac \theta 2+1>0$ by construction and $-\alpha+1+\beta>0$ by Assumption A2.
\end{proof}

Next we show an auxiliary   result useful later on.
\begin{prop}\label{pr: bound for Iu - Iv}
Let Assumptions A1, A2 and A3 hold. Let $u, v \in C_T\mathcal C^{\alpha+1}$. Then for all $\rho\geq1$
\begin{align}\label{eq: bound for Iu - Iv}
\nonumber
\|I(u) - I(v)\|_{C_T \mathcal C^{\alpha+1} }^{(\rho)} \leq & c \|b\|_{L_T^\infty \mathcal C^\beta} \rho^{\frac{\alpha-1-\beta}2 } (1+ \|u\|^2_{C_T \mathcal C^{\alpha+1} } + \|v\|^2_{C_T \mathcal C^{\alpha+1} })^{1/2}\\
& \|u-v\|_{C_T \mathcal C^{\alpha+1} }^{(\rho)}
\end{align} 
where the constant $c$ depends only on  $ L, l$ and $d$. 
\end{prop}
\begin{proof}
Using the definition of $I$ we have

\begin{align*}
\|I(u) - &I(v)\|_{C_T \mathcal C^{\alpha+1} }^{(\rho)} \\
& = \sup_{0\leq t\leq T} e^{-\rho t} \|I_t(u) - I_t(v)\|_{\alpha+1} \\
& = \sup_{0\leq t\leq T} e^{-\rho t} \left\| \int_0^t P_{t-s}\left(  [\mathrm F(\nabla u(s) ) - \mathrm F(\nabla v(s))] b(s)\right)\mathrm d s \right\|_{\alpha+1}. 
\end{align*}
Now using \eqref{eq: mapping Pt Besov spaces} with $\theta = \frac{\alpha+1-\beta}{2} $ (which is positive by Assumption A2)  and \eqref{eq: Bony's estimates} (again by   A2 $\alpha+\beta>0$) we bound the integrand by 
\[
(t-s)^{-\frac{\alpha+1-\beta}{2}}   \|b\|_{L_T^\infty \mathcal C^\beta} \|\mathrm F(\nabla u(s) )- \mathrm F(\nabla v (s)) \|_\alpha
\] 
and using the result of Proposition \ref{pr: mapping prop non linear F} we further bound it by 
\begin{align*}
 c (t-s)^{-\frac{\alpha+1-\beta}{2}}   \|b\|_{L_T^\infty \mathcal C^\beta} \| \nabla u(s)-   \nabla v (s)  \|_\alpha (1+\| \nabla u (s)\|_\alpha^2 + \|\nabla v(s)\|_\alpha^2)^{1/2},
\end{align*}
where the constant $c$ depends on $ L, l$ and $d$. 
Substituting the last bound into the equation above we get
\begin{align*}
\|I(u) - &I(v)\|_{C_T \mathcal C^{\alpha+1} }^{(\rho)} \\
 \leq & c  \|b\|_{L_T^\infty \mathcal C^\beta}  \sup_{0\leq t\leq T} \int_0^t  (t-s)^{-\frac{\alpha+1-\beta}{2}}     e^{-\rho (t-s)}   \\
& e^{-\rho s}   \| \nabla u(s)-   \nabla v (s)  \|_\alpha  (1+\| \nabla u (s)\|_\alpha^2 + \|\nabla v(s)\|_\alpha^2)^{1/2} \mathrm d s  \\
\leq &  c  \|b\|_{L_T^\infty \mathcal C^\beta}  \sup_{0\leq t\leq T} \int_0^t  (t-s)^{-\frac{\alpha+1-\beta}{2}}     e^{-\rho (t-s)} \mathrm d s  \\
&     \| \nabla u-   \nabla v   \|^{(\rho)}_{C_T \mathcal C^\alpha}  (1+\| \nabla u \|_{C_T\mathcal C^\alpha}^2 + \|\nabla v\|_{C_T\mathcal C^\alpha}^2 )^{1/2} .
\end{align*}
Finally we use the bound $\|\nabla f\|_\alpha \leq c\|f\|_{\alpha+1}$ for $f\in \mathcal C^{\alpha+1}$ (which follows from Bernstein inequalities, see e.g. \cite[Lemma 2.1]{bahouri}) and we integrate the singularity since $-\frac{\alpha+1-\beta}{2}>-1 $ to get
\[
c   \|b\|_{L_T^\infty \mathcal C^\beta}  \rho^{\frac{\alpha-1-\beta}{2}} (1+\| u \|_{C_T\mathcal C^{\alpha+1}}^2 + \| v\|_{C_T\mathcal C^{\alpha+1}}^2 )^{1/2}   \|u-v\|_{C_T \mathcal C^{\alpha+1} }^{(\rho)},
\]
as wanted.
\end{proof}
We remark that the power of $\rho$ in \eqref{eq: bound for Iu - Iv} is negative due to Assumption A2 and the idea is to pick $\rho$ large enough so that $I$ is a contraction. However this cannot be done using \eqref{eq: bound for Iu - Iv} directly because of the term 
$ (1+\| u \|_{C_T\mathcal C^{\alpha+1}}^2 + \| v\|_{C_T\mathcal C^{\alpha+1}}^2 )^{1/2}  $. 
Indeed  we are only able to show existence and uniqueness of a solution for a small time-interval  or alternatively for a  small initial condition, as we will see later.
\begin{prop} \label{pr: mapping of J in C}
Let Assumptions A1, A2 and A3 hold. Let $u_0\in \mathcal C^{\alpha+1}$ be given. Then the operator $J$ maps $ C_T\mathcal C^{\alpha+1}$ into itself. In particular, for arbitrary $T, \rho$ and $u \in C_T\mathcal C^{\alpha+1}$ we have
\begin{align}\label{eq: bound for Ju}
\|J(u)\|_{C_T\mathcal C^{\alpha+1}}^{(\rho)} 
 &\leq \|u_0\|_{\alpha+1} \\
& + C \rho^{\frac{\alpha-1-\beta}2} \left(1 + \|u\|^{(\rho)}_{ C_T \mathcal C^{\alpha + 1}} (1 + \|u\|^2_{ C_T \mathcal C^{\alpha + 1}})^{1/2} \right), \nonumber
\end{align}
where $C= c \|b\|_{L_T^\infty \mathcal C^\beta}$ is the constant appearing in \eqref{eq: bound for Iu - Iv} in front of $\rho$ and $c$ depends only on $ L, l$ and $d$. 
\end{prop}
\begin{proof}
It is clear that \eqref{eq: bound for Ju} implies that $J$ maps $ C_T \mathcal C^{\alpha + 1} $ into itself. To prove \eqref{eq: bound for Ju} we use the definition of $J$ to get
\begin{align*}
\|J(u)\|_{ C_T \mathcal C^{\alpha + 1}}^{(\rho)} &= \|P_\cdot u_0 + I(u)\|_{ C_T \mathcal C^{\alpha + 1}}\\
&\leq \| P_\cdot u_0\|_{ C_T \mathcal C^{\alpha + 1}}^{(\rho)} + \|I(u)\|_{ C_T \mathcal C^{\alpha + 1}}^{(\rho)}\\
& =:(A) + (B).
\end{align*}
The term (A) is bounded using the contraction property of $P_t$ in $\mathcal C^\alpha$  and by the definition of the equivalent norm
\[
(A)\leq \|u_0\|_{ C_T \mathcal C^{\alpha + 1}}^{(\rho)} = \sup_{0\leq t\leq T} e^{-\rho t} \|u_0\|_{\alpha+1} =  \|u_0\|_{\alpha+1}.
\]
The term (B) can be bounded similarly as in the proof of Proposition \ref{pr: bound for Iu - Iv} and one gets
\begin{align*}
(B)& \leq c \sup_{0\leq t \leq T} e^{-\rho t } \int_0^t (t-s)^{-\frac{\alpha+1-\beta}2}   \|\mathrm F(\nabla u(s))\|_\alpha \|b(s)\|_\beta \mathrm ds.
\end{align*}
Now we apply Proposition \ref{pr: mapping prop non linear F} with $f=\nabla u(s)$ and $g=0$ to get  
\begin{align*}
\|\mathrm F(\nabla u(s))- \mathrm F(\mathbf 0) + \mathrm F(\mathbf 0)\|_{\alpha}
& \leq \|\mathrm F(\nabla u(s))- \mathrm F(\mathbf 0)\|_\alpha + \|\mathrm F(\mathbf 0)\|_{\alpha}\\
& \leq c + (1+\|\nabla u(s)\|_\alpha^2)^{1/2} \|\nabla u(s)\|_\alpha\\
&\leq c (1+ \|u(s)\|_{\alpha+1}(1+\|u(s)\|_{\alpha+1}^2)^{1/2}).
\end{align*}
Plugging this into (B) we get
\begin{align*}
(B)\leq&   c   \|b\|_{L_T^\infty \mathcal C^\beta } \sup_{0\leq t \leq T}  \int_0^t e^{-\rho (t-s) } (t-s)^{-\frac{\alpha+1-\beta}2}  \mathrm ds\\
&\sup_{0\leq s\leq T } e^{-\rho s } \left(1+ \|u(s)\|_{\alpha+1}(1+\|u(s)\|_{\alpha+1}^2)^{1/2}\right)\\
\leq &  c   \|b\|_{L_T^\infty \mathcal C^\beta } \rho^{\frac{\alpha-1-\beta}2}   \left(1+ \|u\|_{C_T\mathcal C^{\alpha+1}}^{(\rho)} (1+\|u\|_{C_T\mathcal C^{\alpha+1}}^2)^{1/2}\right)
\end{align*}
as wanted.
\end{proof}
Carrying out the same proof in the special case when $F(0)=0$ we easily obtain the result below.
\begin{coroll}\label{cor: mapping of J in C}
Under the assumptions  of Proposition \ref{pr: mapping of J in C} and if moreover  $F(0)=0$ then we have
\begin{equation}\label{eq: bound for Ju for F(0)=0}
\|J(u)\|_{C_T\mathcal C^{\alpha+1}}^{(\rho)} 
\leq \|u_0\|_{\alpha+1} + C \rho^{\frac{\alpha-1-\beta}2} \|u\|^{(\rho)}_{ C_T \mathcal C^{\alpha + 1}} (1 + \|u\|^2_{ C_T \mathcal C^{\alpha + 1}})^{1/2}.
\end{equation}
 
\end{coroll}

To show that $J$ is a contraction in a suitable (sub)space  we introduce a subset of $C_T\mathcal C^{\alpha+1}$ which depends on three parameters, $\rho$, $R$ and $T$. We define
\begin{equation}\label{eq: ball}
B^{(\rho)}_{R, T} := \left\{ f\in C_T\mathcal C^{\alpha+1} \, : \, \|f\|_{C_T\mathcal C^{\alpha+1}}^{(\rho)}\leq 2 R e^{-\rho T} \right\}.
\end{equation}
Now choosing $\rho$, $R$ and $T$ appropriately (depending on the initial condition $u_0$) one can show that $J$ is a contraction by applying Proposition \ref{pr: mapping of J in C} as illustrated below.
\begin{prop}\label{pr: J contraction}
Let Assumptions A1, A2 and A3 hold. Let $R_0$ be a given arbitrary constant. Then there exists $\rho_0$ large enough depending on $R_0$, and $T_0$ small enough depending  on $\rho_0$ such that 
\[
J: B_{R_0,T_0}^{(\rho_0)} \to B_{R_0,T_0}^{(\rho_0)},
\]
for any initial condition   $u_0\in \mathcal C^{\alpha+1}$ such that  $ \|u_0\|_{\alpha+1}\leq R_0$. Moreover for each $u,v \in C_{T_0}\mathcal C^{\alpha+1}$ then
\[
\|J(u)-J(v)\|_{C_{T_0}\mathcal C^{\alpha+1}}^{(\rho_0)} < \|u-v\|_{C_{T_0}\mathcal C^{\alpha+1}}^{(\rho_0)} .
\]
\end{prop}
\begin{proof}
 We begin by taking $u\in B_{R_0,T}^{(\rho)}$ for some arbitrary parameters $ T$ and $\rho$. For this $u$ we have the following bounds
\[
\|u\|_{C_T\mathcal C^{\alpha+1}}^{(\rho)} \leq 2 R_0 e^{-\rho T}
\]
 and 
 \begin{equation}\label{eq: norm for u in B_R}
\|u\|_{C_T\mathcal C^{\alpha+1}} \leq 2 R_0 e^{-\rho T} e^{\rho T} = 2R_0.
 \end{equation}
Let  $u_0\in C^{\alpha+1}$ be such that $\|u_0\|_{\alpha+1}\leq R_0$. Then by Proposition \ref{pr: mapping of J in C}  we obtain 
\begin{align*}
\|J(u)\|_{C_T\mathcal C^{\alpha+1}}^{(\rho)} 
&\leq R_0 + C \rho^{\frac{\alpha-1-\beta}2} \left(1 +  2R_0 e^{-\rho T}(1 + 4R_0^2 )^{1/2} \right)\\
& = R_0 e^{-\rho T} \left( e^{\rho T}  +  \frac C {R_0} \rho^{\frac{\alpha-1-\beta}2} e^{\rho T}  + 2C\rho^{\frac{\alpha-1-\beta}2}(1 + 4R_0^2 )^{1/2} \right).
\end{align*}
To show that $J(u)\in B_{R_0,T}^{(\rho)}$ we need to pick  $\rho_0$ and $T_0 $ such that 
\begin{equation}\label{eq: bound to get a contraction}
e^{\rho T}  +  \frac C {R_0} \rho^{\frac{\alpha-1-\beta}2} e^{\rho T}  + 2C \rho^{\frac{\alpha-1-\beta}2}(1 + 4R_0^2 )^{1/2} \leq 2.
\end{equation} 
 This is done as follows. First we pick $\rho_0\geq 1 $ depending on $R_0$ and large enough such that the following three conditions hold
\begin{eqnarray}
2C \rho_0^{\frac{\alpha-1-\beta}2}(1 + 4R_0^2 )^{1/2} \leq \frac14 \label{eq: bound 1}\\
\frac C {R_0} \rho_0^{\frac{\alpha-1-\beta}2} \leq \frac14 \label{eq: bound 2}\\
C\rho_0^{\frac{\alpha-1-\beta}2} (1 + 8R_0^2 )^{1/2} <1.\label{eq: bound 3}
\end{eqnarray}
This is always possible since $\rho\mapsto \rho^{\frac{\alpha-1-\beta}2}$ is decreasing. Moreover this can be done independently of $T$. We also remark that the third bound is not needed to show that $J(u)\in B_{R_0,T}^{(\rho)}$ but will be needed below to show that $J$ is a contraction for the chosen set of parameters $R_0, \rho_0, T_0$.\\
Next we pick $T_0>0$ depending on $\rho_0, R_0$ and small enough such that 
\begin{equation}\label{eq: bound 4}
e^{\rho_0 T_0}\leq 1+\frac25.
\end{equation}
This is always possible since $T\mapsto e^{\rho_0 T}$ is increasing,  continuous and has minimum 1 at 0. \\
With these parameters, \eqref{eq: bound to get a contraction} is satisfied under the assumptions \eqref{eq: bound 1}, \eqref{eq: bound 2} and \eqref{eq: bound 4}. Indeed
\begin{equation*}
e^{\rho_0 T_0}  +  \frac C {R_0} \rho^{\frac{\alpha-1-\beta}2} e^{\rho_0 T_0}  + 2C \rho_0^{\frac{\alpha-1-\beta}2}(1 + 4R_0^2 )^{1/2} \leq 1+\frac25  +\frac 14(1+\frac25)+ \frac14 = 2.
\end{equation*}
It is left to prove that $J$ is a contraction on $B^{(\rho_0)}_{R_0, T_0}$. For this, it is enough to use Proposition \ref{pr: bound for Iu - Iv} for $u,v\in B^{(\rho_0)}_{R_0, T_0}$   
\begin{align*}
\|I(u)-I(v)\|_{C_{T_0}\mathcal C^{\alpha+1}}^{(\rho_0)} &\leq 
C \rho_0^{\frac{\alpha-1-\beta}2} (1+ 2(2R_0)^2  )^{1/2}\|u-v\|_{C_{T_0}\mathcal C^{\alpha+1}}^{(\rho)}\\
& < \|u-v\|_{C_{T_0}\mathcal C^{\alpha+1}}^{(\rho_0)},
\end{align*}
 where the last bound is ensured by \eqref{eq: bound 3}.
\end{proof}

Using the last result we can show that a unique solution exists locally (for small time $T_0$) in the whole space $C_{T_0}\mathcal C^{\alpha+1}$.
\begin{theorem}\label{thm: local fixed point for J}
Let Assumptions A1, A2 and A3 hold. 
Let  $u_0\in \mathcal C^{\alpha +1}$ be given. Then there exists a unique local mild solution $u$ to \eqref{eq: mild solution} in $C_{T_0}\mathcal C^{\alpha+1}$, where $T_0 $ is small enough and it is chosen as in Proposition \ref{pr: J contraction} (depending on the norm of $u_0$).
 \end{theorem}
\begin{proof}
Let $R_0=\|u_0\|_{\alpha+1}$ and $\rho_0$ and $T_0$ such that \eqref{eq: bound 1}--\eqref{eq: bound 4} are  satisfied. \\
\emph{Existence.}
By Proposition \ref{pr: J contraction} we know that the mapping $J$ is a contraction on $ B^{(\rho_0)}_{R_0, T_0}$   and so there exists a solution $u \in B^{(\rho_0)}_{R_0, T_0}$ which is unique in the latter subspace. \\
\emph{Uniqueness.}
Suppose that there are two solutions $u_1$ and $u_2$ in $C_{T_0}\mathcal C^{\alpha+1}$. Then obviously $u_i=  J(u_i)$ and $ \| u_i\|_{C_{T_0}\mathcal C^{\alpha+1}}< \infty$ for $i=1,2$. We set $r:= \max\{\| u_i\|_{C_{T_0}\mathcal C^{\alpha+1}}, i=1,2 \}$ (which only depends on $u_i$ and not on any $\rho$). 
By Proposition \ref{pr: bound for Iu - Iv} for any $\rho\geq 1$ we have that the $\rho$-norm of the difference $u_1-u_2$ is bounded by
\begin{align*}
\|u_1-&u_2\|_{C_{T_0}\mathcal C^{\alpha+1}}^{(\rho)} = 
\|I(u_1)-I(u_2)\|_{C_{T_0}\mathcal C^{\alpha+1}}^{(\rho)} \\
& \leq C \rho^{\frac{\alpha-1-\beta}2 } (1+ \|u_1\|^2_{C_{T_0} \mathcal C^{\alpha+1} } + \|u_2\|^2_{C_{T_0} \mathcal C^{\alpha+1} })^{1/2}\|u_1-u_2\|_{C_{T_0} \mathcal C^{\alpha+1} }^{(\rho)}\\
& \leq C \rho^{\frac{\alpha-1-\beta}2 } (1+ 2r^2)^{1/2}\|u_1-u_2\|_{C_{T_0} \mathcal C^{\alpha+1} }^{(\rho)}.
 \end{align*}
Choosing $\rho_0$ large enough such that $1- C \rho_0^{\frac{\alpha-1-\beta}2 } (1+ 2r^2)^{1/2} >0$ implies that $\|u_1-u_2\|_{C_{T_0} \mathcal C^{\alpha+1} }^{(\rho_0)} \leq 0$ and hence the difference must be 0 in the space $C_{T_0}\mathcal C^{\alpha+1}$, thus $u_1=u_2$.
 \end{proof}
 
\begin{rem}\label{rm:uniqueness1}
Note that in the proof of uniqueness of Theorem \ref{thm: local fixed point for J} we do not assume anything about the size of time $T_0$. Hence, if a solution to  \eqref{eq: mild solution} exists up to time $T$ in the space $C_T\mathcal C^{\alpha+1}$, then it is unique.
\end{rem}

An alternative existence and uniqueness result is shown below. A global in time solution is found  up to  any given time $T$, but in this case we have to restrict the choice of initial conditions $u_0$ to a set with  small norm (depending on $T$). Moreover we are able to show this result only under the extra condition that $F(0)=0$.

\begin{prop}\label{pr: fixed point for J for any T}
Let Assumptions A1, A2 and A3 hold. Assume $F(0)=0$. Let $T>0$ be given and arbitrary. Then there exists $\rho_0$ large enough such that for all $u_0\in B_{\frac12, T}^{(\rho_0)}$ then
\begin{equation}\label{eq: fixed point for J}
J: B_{1, T}^{(\rho_0)} \to B_{1, T}^{(\rho_0)}
\end{equation}
and $J$ is a contraction on $B_{1, T}^{(\rho_0)}$, namely for $u,v \in  B_{1, T}^{(\rho_0)} $ we have 
\begin{equation}\label{eq: fixed point for J bound}
\|J(u)-J(v)\|^{(\rho_0)}_{C_T\mathcal C^{\alpha+1}}< \|u-v\|^{(\rho_0)}_{C_T\mathcal C^{\alpha+1}}. 
\end{equation}
\end{prop}
\begin{proof}
We recall that for some given $R, \rho$ and $T$, the assumption $u_0\in B_{R, T}^{(\rho)} $ means that $ \|u_0\|^{(\rho)}_{C_T\mathcal C^{\alpha+1}}\leq 2Re^{-\rho T}$, see \eqref{eq: ball}.
Moreover $u_0$ does not depend on time hence $ \|u_0\|^{(\rho)}_{C_T\mathcal C^{\alpha+1}} = \|u_0\|_{\alpha+1}$
so  $u_0\in B_{\frac12, T}^{(\rho)}$ implies  
\begin{equation*}
\|u_0\|_{\alpha+1}\leq e^{-\rho T}. 
\end{equation*}
Using this and  Corollary \ref{cor: mapping of J in C} we have 
\begin{align*}
\|J(u)\|^{(\rho)}_{C_T\mathcal C^{\alpha+1}} 
&\leq \|u_0\|_{\alpha+1} + C \rho^{\frac{\alpha-1-\beta}2} \|u\|^{(\rho)}_{ C_T \mathcal C^{\alpha + 1}} (1 + \|u\|^2_{ C_T \mathcal C^{\alpha + 1}})^{1/2}\\
&\leq  e^{-\rho T} + C \rho^{\frac{\alpha-1-\beta}2} \|u\|^{(\rho)}_{ C_T \mathcal C^{\alpha + 1}} (1 + \|u\|^2_{ C_T \mathcal C^{\alpha + 1}})^{1/2}.
\end{align*}
Let $u\in B_{1,T}^{(\rho)}$. Then  $\|u\|^{(\rho)}_{C_T\mathcal C^{\alpha+1}}\leq 2e^{-\rho T}$ and 
\begin{equation} \label{eq: norm of u in B}
\|u\|_{C_T\mathcal C^{\alpha+1}}\leq 2.
\end{equation}
 Thus the bound above becomes
\begin{align*}
\|J(u)\|^{(\rho)}_{C_T\mathcal C^{\alpha+1}} 
&\leq e^{-\rho T} + C  \rho^{\frac{\alpha-1-\beta}2} 2e^{-\rho T} (1 + 4)^{1/2}\\
& = 2 e^{-\rho T} (\frac12 + C\sqrt 5 \rho^{\frac{\alpha-1-\beta}2} ).
\end{align*}
We choose $\bar \rho_0$  such that $\frac12 + C\sqrt 5 \bar \rho_0^{\frac{\alpha-1-\beta}2} =1 $, and since the function   $\rho \mapsto \rho^{\frac{\alpha-1-\beta}2} $ is decreasing, for each $\rho_0\geq \bar\rho_0$ we have 
\begin{equation}\label{eq: rho not bar}
\frac12 + C\sqrt 5  \rho_0^{\frac{\alpha-1-\beta}2} \leq 1.
\end{equation}
Then for $\rho=\rho_0$ we have $\|J(u)\|^{(\rho_0)}_{C_T\mathcal C^{\alpha+1}} \leq 2 e^{-\rho_0 T} $ which implies that $J(u)\in B_{1, T}^{(\rho_0)}$ and this shows \eqref{eq: fixed point for J}.

To show \eqref{eq: fixed point for J bound}, let $u,v\in B_{1,T}^{(\rho_0)}\subset C_T\mathcal C^{\alpha+1}$ with $\rho_0\geq \bar\rho_0$. Then by Proposition \ref{pr: bound for Iu - Iv} and by \eqref{eq: norm of u in B} 
\begin{align*}
\|J(u)-&J(v)\|_{C_T\mathcal C^{\alpha+1}}^{(\rho_0)} \\
& \leq C \rho_0^{\frac{\alpha-1-\beta}{2}} \left( 1+ \|u\|_{C_T\mathcal C^{\alpha+1}}^2 + \|v\|_{C_T\mathcal C^{\alpha+1}}^2 \right)^{1/2}  \|u-v\|_{C_T\mathcal C^{\alpha+1}}^{(\rho_0)}\\
& \leq C \rho_0^{\frac{\alpha-1-\beta}{2}} \left( 1+ 4+4 \right)^{1/2}  \|u-v\|_{C_T\mathcal C^{\alpha+1}}^{(\rho_0)}\\
& \leq 3C \rho_0^{\frac{\alpha-1-\beta}{2}} \|u-v\|_{C_T\mathcal C^{\alpha+1}}^{(\rho_0)}.
\end{align*}
We now chose $\rho_0\geq \bar \rho_0$ large enough so that 
\begin{equation}\label{eq: rho not}
3C \rho_0^{\frac{\alpha-1-\beta}{2}} <1
\end{equation}
 and the proof is concluded.
\end{proof}

\begin{theorem}\label{thm: fixed point for J for small u0}
Let Assumptions A1, A2 and A3 hold.  Let $T>0$ be given and let $F(0)=0$. Then there exists $\delta>0$ depending on $T$ such that for each $u_0 $ with $\|u_0\|_{\alpha+1}\leq \delta$  there exists a unique solution $u\in C_T\mathcal C^{\alpha +1} $ to \eqref{eq: mild solution}.
\end{theorem}
\begin{proof}
\emph{Existence.}
We choose $\rho_0$ according to \eqref{eq: rho not} and \eqref{eq: rho not bar}. Let $\delta = e^{-\rho_0 T}$. Then the assumption $\|u_0\|_{\alpha+1}\leq \delta$ means $u_0\in B_{\frac12, T}^{(\rho_0)}$ and by Proposition \ref{pr: fixed point for J for any T} we know that the mapping $J$ is a contraction on $ B^{(\rho_0)}_{1, T}$. Thus there exists a unique fixed point $u$ in $ B^{(\rho_0)}_{1, T}$  which is a solution. \\
\emph{Uniqueness.} 
This is shown like in the uniqueness proof of Theorem \ref{thm: local fixed point for J}, with $T$ instead of $T_0$.
\end{proof}
\begin{rem}\label{rm:uniqueness2}
Note that in the proof of uniqueness of Theorem \ref{thm: local fixed point for J} we do not actually use the assumption   $\|u_0\|_{\alpha+1}\leq \delta$, so if $F(0)=0$ then  uniqueness holds for any initial condition and any time $T$, when a solution exists. 
\end{rem}

We now show continuity of the solution $u$ with respect to the initial condition $u_0$. This is done in the following proposition both   for the case of existence and uniqueness of a solution $u$ for an arbitrary initial condition and a sufficiently  small time $T_0$ (Theorem \ref{thm: local fixed point for J}) and for the case of existence and uniqueness of a solution $u$  for an arbitrary  time $T$ and for a sufficiently small (in norm)  initial condition $u_0$ (Theorem \ref{thm: fixed point for J for small u0}).

\begin{prop}\label{pr: continuity wrt u0}
\begin{itemize}
\item[(i)] Let the assumptions of Theorem \ref{thm: local fixed point for J} hold and let  $R_0>0$ be arbitrary and fixed. Let $u$ be the unique solution  found in Theorem \ref{thm: local fixed point for J} on $[0,T_0] $ with initial condition $u_0$ such that $\|u_0\|\leq R_0$ and where $T_0$ depends on $R_0$. Then $u$ is continuous with respect to the initial condition $u_0$, namely
\[
\|u\|^{(\rho_0)}_{C_{T_0}\mathcal C^{\alpha+1}} \leq 2 \|u_0\|_{\alpha+1}
\]
for $\rho_0$ large enough.
\item[(ii)]  Let the assumptions of Theorem \ref{thm: fixed point for J for small u0} hold and let  $T>0$ be arbitrary and fixed.    Let $u$ be the unique solution  found in Theorem    \ref{thm: fixed point for J for small u0}  on $[0,T]$ with initial condition $u_0$ such that $\|u_0\|\leq e^{-\rho_0 T} $ for $\rho_0$ large enough. Then the unique solution $u$ is continuous with respect to the initial condition $u_0$, namely
\[
\|u\|^{(\rho_0)}_{C_T\mathcal C^{\alpha+1}} \leq 2 \|u_0\|_{\alpha+1}.
\] 
\end{itemize}
\end{prop}

\begin{proof}
\emph{(i)} Let $\rho_0$ be chosen according to \eqref{eq: bound 1}--\eqref{eq: bound 3} and $T_0$ according to \eqref{eq: bound 4}. Take $u_0$ such that $\|u_0\|_{\alpha+1}\leq R_0$. Then by Proposition  \ref{pr: J contraction}  we have $ J: B_{R_0, T_0}^{(\rho_0)}\to B_{R_0, T_0}^{(\rho_0)}$ and so by  \eqref{eq: norm for u in B_R} the unique solution $u$ given in Theorem \ref{thm: local fixed point for J} satisfies $\|u\|_{C_{T_0}\mathcal C^{\alpha+1}} \leq 2R_0$ for any initial conditions $u_0$ with $\|u_0\|_{\alpha+1}\leq R_0$. 
 Using this and Corollary \ref{cor: mapping of J in C} we have
\begin{align*}
\|u\|^{(\rho_0)}_{C_{T_0}\mathcal C^{\alpha+1}}
 &= \|J(u)\|^{(\rho_0)}_{C_{T_0}\mathcal C^{\alpha+1}} \\
 &\leq  \|u_0\|_{\alpha+1} + C \rho_0^{\frac{\alpha-1-\beta}{2}} \|u\|^{(\rho_0)}_{C_{T_0}\mathcal C^{\alpha+1}} (1+ \|u\|_{C_{T_0}\mathcal C^{\alpha+1}})^{1/2}\\
 &\leq \|u_0\|_{\alpha+1} + \sqrt{1+4R_0^2}  C \rho_0^{\frac{\alpha-1-\beta}{2}} \|u\|^{(\rho_0)}_{C_T\mathcal C^{\alpha+1}} .
\end{align*}
By the choice of $\rho_0$ according to \eqref{eq: bound 1} we have $2\sqrt{1+4R_0^2}  C \rho_0^{\frac{\alpha-1-\beta}{2}} \leq \frac14$ hence
\[
\|u\|^{(\rho_0)}_{C_{T_0}\mathcal C^{\alpha+1}} \leq \|u_0\|_{\alpha+1} + \frac12 \|u\|^{(\rho_0)}_{C_T\mathcal C^{\alpha+1}} ,
\] 
and rearranging terms we conclude.

\emph{(ii)} Let $\rho_0$ be chosen according to \eqref{eq: rho not bar}. Then for all $u_0\in B_{\frac12, T}^{(\rho_0)}$ (that is for $\|u_0\|_{\alpha+1}\leq e^{-\rho_0 T}$) we have $J: B_{1, T}^{(\rho_0)} \to B_{1, T}^{(\rho_0)}$ by Proposition \ref{pr: fixed point for J for any T}. In particular, the unique solution $u$ given in Theorem \ref{thm: fixed point for J for small u0} belongs to  $B_{1, T}^{(\rho_0)}$, and \eqref{eq: norm of u in B} holds, that is $\|u\|_{C_T\mathcal C^{\alpha+1}}\leq 2$. Using this and Corollary \ref{cor: mapping of J in C}  we have
\begin{align*}
\|u\|^{(\rho_0)}_{C_T\mathcal C^{\alpha+1}}
 &= \|J(u)\|^{(\rho_0)}_{C_T\mathcal C^{\alpha+1}} \\
 &\leq  \|u_0\|_{\alpha+1} + C \rho_0^{\frac{\alpha-1-\beta}{2}} \|u\|^{(\rho_0)}_{C_T\mathcal C^{\alpha+1}} (1+ \|u\|_{C_T\mathcal C^{\alpha+1}})^{1/2}\\
 &\leq \|u_0\|_{\alpha+1} + \sqrt 5 C \rho_0^{\frac{\alpha-1-\beta}{2}} \|u\|^{(\rho_0)}_{C_T\mathcal C^{\alpha+1}} .
\end{align*}
By the choice of $\rho_0$ according to \eqref{eq: rho not bar} we have $\sqrt 5 C \rho_0^{\frac{\alpha-1-\beta}{2}} \leq \frac12$ and  we conclude as in part (i). 
\end{proof}

Finally we conclude this section by investigating the blow-up for the  solution $u$ to the PDE. It is still an open problem to show whether the solution $u$ blows up or not, but we have the following result that states that if blow-up occurs, then it does so in finite time. 

\begin{prop}\label{pr: blow up}
Let $u_0\in \mathcal C^{\alpha +1}$ and $T>0$ be given. 
Then one of the following statements holds:
\begin{itemize}
\item[(a)] There exists a time $t^*\in [0,T]$ such that $\lim_{s\to t^*}\|u(s)\|_{\alpha+1}=\infty$; Or
\item[(b)] there exists a solution $u$ for all $t\in[0,T]$. 
\end{itemize}
\end{prop}

\begin{proof}
Assume that  $\limsup_{s\to t^*} \|u(s)\|_{\alpha +1} = \infty$ for some $t^*\in [0,T]$. Suppose  moreover by contradiction that $\liminf_{s\to t^*}\|u(s)\|_{\alpha+1}<\infty$. Then we can find $R_0>0$ and a sequence $t_k\to t^*$ such that  $\|u(t_k)\|_{\alpha+1} <R_0$ for all $k$. Let us now restart the PDE from $u(t_k)$ and apply Theorem \ref{thm: local fixed point for J}: We know that there exists a solution for the interval $[t_k, t_k+T_0]$, where $T_0>0$ depends on $R_0$ but not on $k$. Thus we are able to extend the solution $u$ past $t^*$ because as $k\to\infty$ we have $t_k+T_0\to t^*+T_0$. Thus it cannot be that $\limsup_{s\to t^*} \|u(s)\|_{\alpha +1} = \infty$ and $\liminf_{s\to t^*}\|u(s)\|_{\alpha+1}<\infty$ for some $t^*\in [0,T]$. This means that if $\limsup_{s\to t^*} \|u(s)\|_{\alpha +1} = \infty $ for some $t^*\in[0,T]$ then actually also $\lim_{s\to t^*} \|u(s)\|_{\alpha +1} = \infty$, which is case (a).
Otherwise, if  $\limsup_{s\to t^*}\|u(s)\|_{\alpha+1}<\infty$ for all  $t^*\in [0,T]$
then a global solution on $[0,T]$ must exists, which is case (b).
\end{proof}

Further research is needed to show either global in time solution or the existence of a finite blow-up time. The difficulty here is due to the quadratic non-linearity and the fact that this term is multiplied by the distributional coefficient. This prevents us to apply classical techniques such as the Cole-Hopf transformation which would be used in the special case $F(x)=x^2$ and $b\equiv 1$ to linearise the equation.

\section{A global existence result}\label{sc:global}
In this section we provide a global result on existence and uniqueness of a solution upon imposing further assumptions on the non-linearity $F$.  In particular, we will exclude the quadratic case but still allow for a rich class of non-linear functions. 
\begin{description}
\item[A4] \textbf{Further assumption on non-linear term $F$.}
\emph{Let $F: \mathbb R^d \to \mathbb R$ be globally Lipschitz, i.e.,  there exists a positive constant $L$ such that  for all $x, y\in\mathbb R^d$ we have
\[
|F(x)-F(y)|\leq \tilde L |x-y|_d.
\]}
\end{description}
Assumption A4 implies that $F$ has sub-linear growth, that is, there exists a positive constant $\tilde l$ such that for all $x\in \mathbb R^d$ 
\[
|F(x)|\leq \tilde l (1+ |x|_d).
\]
Moreover also  the  operator $\mathrm F: \mathcal C^\alpha \to \mathcal C^\alpha$ has sub-linear growth in $\mathcal C^\alpha$, namely there exists $c>0$ such that for all $f\in\mathcal C^\alpha$ we have 
\begin{equation}\label{eq:Fsublinear}
\|\mathrm F(f)\|_\alpha \leq c (1+\|f\|_\alpha). 
\end{equation}
Indeed 
\begin{align*}
\|\mathrm F(f)\|_\alpha 
& = \sup_{x\in\mathbb R^d} |Ff(x)| + \sup_{x\in\mathbb R^d} \sup_{|y|_d\leq 1} \frac{|Ff(x+y) - Ff(x)|}{|y|_d^\alpha}\\
&\leq \sup_{x\in\mathbb R^d} \tilde l (1+|f(x)) +\sup_{x\in\mathbb R^d}  \sup_{|y|_d\leq 1} \frac{\tilde L |f(x+y) - f(x)|}{|y|_d^\alpha}\\
& \leq c (1+\|f\|_\alpha). 
\end{align*}
This extra assumption allows us to find a priori bounds on the solution, as follows.

\begin{prop}[A priori bounds]\label{pr:priori}
Let Assumptions A1, A2, A3 and A4 hold.  Let $T<\infty$ be an arbitrary time and $u_0\in \mathcal C^{\alpha+1}$.
If there exists $u\in C_T\mathcal C^{\alpha+1}$ such that  
\begin{equation}\label{eq:lu}
u(t) = \lambda P_t u_0 + \lambda \int_0^t P_{t-r} (F(\nabla u(r)) b(r)) \ud r
\end{equation}
where $\lambda \in [0,1]$ is fixed, then for all $t\in[0,T]$ it must hold
\[
\| u (t)\|_{\alpha+1} \leq K
\]
for some finite constant $K$ which depends only on $T, b$ and $u_0$. In particular, $\|u\|_{C_T \mathcal C^{\alpha+1}}\leq K$.
\end{prop}
Note that when $\lambda=1$ then \eqref{eq:lu} reduces to \eqref{eq: mild solution}. By slight abuse of notation, in this result we use $u$ for the solution of \eqref{eq:lu} for $\lambda\in[0,1]$.
\begin{proof}
Let $u\in C_T\mathcal C^{\alpha+1}$ be a solution of \eqref{eq:lu}, that is 
\begin{equation}\label{eq:e1}
u(t) = \lambda P_t u_0 + \lambda I_t(u).
\end{equation}
 Note that $\mathrm F(\nabla u)\in C_T \mathcal C^{\alpha+1}\subset L^\infty_T \mathcal C^{\alpha+1}$ and so $I(u) \in C_T \mathcal C^{\alpha+1}$ by Lemma  \ref{lm: continuity of I} and by Assumption A3. Now we apply \eqref{eq: mapping Pt Besov spaces} and assumption A4 to get 
 \begin{align*}
 \|I_t(u)\|_{\alpha+1} 
 &\leq \int_0^t \| P_{t-s} (F(\nabla u (s)) b(s) \|_{\alpha+1} \ud s \\
& \leq  c \| b \|_{L^\infty_T\mathcal C^\beta}  \int_0^t ({t-s})^{-\frac{\alpha+1-\beta}2}(1+ \|\nabla u (s) \|_{\alpha}) \ud s \\
& \leq  c \| b \|_{L^\infty_T\mathcal C^\beta}  \int_0^t ({t-s})^{-\frac{\alpha+1-\beta}2}(1+ \| u (s) \|_{\alpha+1}) \ud s .
 \end{align*}
Taking the $\mathcal C^{\alpha+1}$ norm of \eqref{eq:e1} and plugging the above estimate in, we obtain
\begin{align*}
 \| u(t)\|_{\alpha+1} 
\leq &  \lambda \|P_t u_0\|_{\alpha+1} + \lambda \|I_t( u)\|_{\alpha+1}\\
 \leq & c \| u_0\|_{\alpha+1} + c \| b \|_{L^\infty_T\mathcal C^\beta}  T^{\frac{-\alpha+1+\beta}2} \\
 & + c \| b \|_{L^\infty_T\mathcal C^\beta}
 \int_0^t  ({t-s})^{-\frac{\alpha+1-\beta}2}\| u (s) \|_{\alpha+1} \ud s . 
\end{align*}
Now an application of Gronwall's lemma and the evaluation of the supremum over $t\in[0,T]$ allows to conclude. 
\end{proof}

Our strategy to show global existence of a solution of \eqref{eq: mild solution} is to apply Schaefer's fixed point theorem. To this aim, for $\eps>0$   let us define   the space $ C_T^\eps\mathcal C^{\alpha+1} $ as the collection of all functions $f:[0,T]\times \mathbb R^d \to \mathbb R$ with finite $\| \cdot \|_{\eps, \alpha+1}$ norm, where the latter  is given by 
\[
\|f\|_{\eps,{\alpha+1}}: = \sup_{0\leq t \leq T} \|f(t)\|_{\alpha+1} + \sup_{0\leq s< t \leq T} \frac{\|f(t)-f(s)\|_{\alpha+1} }{(t-s)^\eps}.
\]
In order to  apply Schaefer's fixed point theorem, it is convenient to work in $ C_T^\eps\mathcal C^{\alpha+1}$ rather than $C_T\mathcal C^{\alpha+1}$, the reason being that balls in $ C_T^{\eps'}\mathcal C^{\alpha'+1}$ are pre-compact sets in $ C_T^\eps\mathcal C^{\alpha+1}$ for $\eps'>\eps$ and $\alpha'>\alpha$. 

For ease of reading we set
\[
G_r(u) := \mathrm F(\nabla u(r)) b(r).
\]
Using Assumption A4  and \eqref{eq: Bony's estimates} we have that for $u(r)\in \mathcal C^{\alpha+1}$ then
\begin{equation}\label{eq:G}
\|G_r(u)\|_\beta \leq  c (1+ \| u(r)\|_{\alpha+1} ),
\end{equation}
where $c$ depends on $b$ and $\tilde l$. 
Moreover by Proposition \ref{pr: mapping prop non linear F} we have that for $u(r), v(r) \in \mathcal C^{\alpha+1}$ then
\begin{equation}\label{eq:Gdiff}
\|G_r(u)-G_r(v)\|_\beta \leq  c (1+ \| u(r)\|^2_{\alpha+1}+\| v(r)\|^2_{\alpha+1} )^{1/2}  \|u(r)-v(r) \|_{\alpha+1},
\end{equation}
where $c$ depend on $b, l, L$ and $d$.

We now state and prove three preparatory results that are the keys steps needed to apply Schaefer's fixed point theorem.

\begin{lemma}\label{lm:J1}
 Let Assumptions A1, A2 and A3 hold and fix $\eps> 0$ such that $\alpha -1 -\beta +\eps <0$. Let $u_0\in \mathcal C^{\alpha+1+ 2\eps +\nu}$ for some small $\nu>0 $.  
If $u\in  C_T \mathcal C^{\alpha+1} $ then  $J(u)\in  C_T^{\eps'}\mathcal C^{\alpha'+1} $ for some $\eps'>\eps$ and $\alpha'>\alpha$, and 
 \begin{equation}\label{eq:J1}
 \|J(u)\|_{\eps', \alpha'+1} \leq c\|u_0\|_{\alpha+1+2\eps +\nu} +c T^{\frac{-\alpha'+1+\beta-2\eps'}2}(1+ \|u\|_{C_T \mathcal C^{\alpha+1}}).
 \end{equation}
\end{lemma}
\begin{rem}
Note that the parameter $\eps$ in   Lemma \ref{lm:J1} could in principle betaken equal zero, in which case we would only need $u_0\in \mathcal C^{\alpha+1+\nu}$ and $\eps'>0$. Later on however, $\eps$ will be chosen strictly greater than zero, hence we state and prove the result for $\eps>0$.
\end{rem}
\begin{proof}[Proof of Lemma \ref{lm:J1}]
First we note that it is always possible to pick $\eps >0$ such that $\alpha -1 -\beta +\eps <0$, because  $\alpha -1 -\beta <0$ by assumption A2. Let $u \in C_T \mathcal C^{\alpha+1}$. Moreover let us pick any $\alpha'>\alpha$ small enough such that $\alpha' -1 -\beta <0$ and $\alpha'+1 < \alpha+1+\nu$. Then we can easily see that for all $t\in[0,T]$ we have $J_t(u)\in \mathcal C^{\alpha'+1}$ as follows. 
\begin{align*}
\|J_t(u)\|_{\alpha'+1}
= &  \|P_t u_0 + \int_0^t P_{t-r} G_r(u) \ud r \|_{\alpha'+1}\\\nonumber
\leq &  \|P_t u_0\|_{\alpha'+1} + \int_0^t \| P_{t-r} G_r(u) \|_{\alpha'+1}  \ud r\\\nonumber
\leq &  c\| u_0\|_{\alpha'+1} + \int_0^t (t-r)^{-\frac{\alpha'+1-\beta}2} \|G_r(u) \|_{\beta}  \ud r\\\nonumber
\leq &  c\| u_0\|_{\alpha'+1} +  c T^{\frac{-\alpha'+1+\beta}2} (1+\|u\|_{C_T \mathcal C^{\alpha+1}}),
\end{align*}
where we have used \eqref{eq: mapping Pt Besov spaces} and \eqref{eq: mapping Pt-I Besov spaces} in the second inequality, and \eqref{eq:G} in the last inequality. Note that $-\alpha'+1+\beta>0$ by construction, and $u_0\in   \mathcal C^{\alpha+1+2\eps+\nu}\subset  \mathcal C^{\alpha'+1}$.

In order to show that $J(u)\in  C_T^{\eps'}\mathcal C^{\alpha'+1}$ we need to control the $\eps'$-H\"older semi-norm. We now choose $\eps'>\eps$ small enough such that  $\alpha' -1 -\beta +2\eps'<0$ and $\alpha' +1 +2\eps' < \alpha + 1 + 2\eps + \nu$, which is always possible. Then $u_0\in \mathcal C^{\alpha'+1+2\eps'}$ and  we express the difference $J_t(u)- J_s(u)$ for all $0\leq s <t\leq T$ as
\begin{align}\label{eq:Jint}
\|J_t(u)- J_s(u)\|_{\alpha'+1} 
\leq  &  \|(P_{t-s}- I) (P_s u_0)\|_{\alpha'+1} \\ \nonumber
&+ \|\int_0^s (P_{t-s}- I) (P_{s-r} G_r(u))\ud r\|_{\alpha'+1}\\ \nonumber
&+ \|\int_s^t P_{t-r}  G_r(u) \ud r\|_{\alpha'+1}\\ \nonumber
=&: M_1 + M_2 + M_3. 
\end{align}
Using \eqref{eq: mapping Pt Besov spaces} we get for the first term
\[
M_1 \leq (t-s)^{\eps'} \|P_s u_0\|_{\alpha'+1+2\eps'} \leq c (t-s)^{\eps'} \|u_0\|_{\alpha'+1+2\eps'},
\]
and $u_0\in   \mathcal C^{\alpha+1+2\eps+\nu}\subset  \mathcal C^{\alpha'+1+2\eps'}$ by choice of $\alpha'$ and $\eps'$.\\
The second term can be bounded using \eqref{eq: mapping Pt Besov spaces}, \eqref{eq: mapping Pt-I Besov spaces} and \eqref{eq:G}, and produces a singularity integrable in time by choice of the parameters. We get
\begin{align*}
M_2 
& \leq  \int_0^s  (t-s)^{\eps'}\| P_{s-r} G_r(u)\|_{\alpha'+1+2\eps'} \ud r\\
&  \leq (t-s)^{\eps'} s^{\frac{-\alpha'+1+\beta-2\eps'}2} c (1 + \|u\|_{C_T \mathcal C^{\alpha+1}})\\
&  \leq (t-s)^{\eps'} T^{\frac{-\alpha'+1+\beta-2\eps'}2} c (1 + \|u\|_{C_T \mathcal C^{\alpha+1}}).
\end{align*}
The third term is similar, and using  \eqref{eq: mapping Pt Besov spaces} and \eqref{eq:G} we obtain
\begin{align*}
M_3 
&\leq\int_s^t   (t-r)^{-\frac{\alpha'+1-\beta}2} \| G_r(u)\|_{\alpha'+1} \ud r\\
&\leq\int_s^t   (t-r)^{-\frac{\alpha'+1-\beta}2} \| G_r(u)\|_{\alpha'+1} \ud r\\
&\leq (t-s)^{\eps'} (t-s)^{\frac{-\alpha'+1+\beta-2\eps'}2}  c (1 + \|u\|_{C_T \mathcal C^{\alpha+1}})\\
&\leq (t-s)^{\eps'} T^{\frac{-\alpha'+1+\beta-2\eps'}2}  c (1 + \|u\|_{C_T \mathcal C^{\alpha+1}}).
\end{align*}
Putting everything together we get 
\begin{align*}
\|J(u)\|_{\eps', \alpha'+1} 
=& \sup_{0\leq t\leq T} \|J_t(u)\|_{\alpha'+1} + \sup_{0\leq s< t\leq T} \frac{\|J_t(u)- J_s(u)\|_{\alpha'+1}}{ (t-s)^{\eps'}}\\
 \leq&     c\| u_0\|_{\alpha'+1} +  c T^{\frac{-\alpha'+1+\beta}2} (1+\|u\|_{C_T \mathcal C^{\alpha+1}})\\
 & + c \|u_0\|_{\alpha'+1+2\eps'} + 2 T^{\frac{-\alpha'+1+\beta-2\eps'}2}  c (1 + \|u\|_{C_T \mathcal C^{\alpha+1}})\\
 \leq &c\|u_0\|_{\alpha+1+2\eps+\nu} +c T^{\frac{-\alpha'+1+\beta-2\eps'}2}(1+ \|u\|_{C_T \mathcal C^{\alpha+1}}),
\end{align*}
and the proof is complete. 
\end{proof}

\begin{rem}\label{rm:holder}
Applying  Lemma \ref{lm:J1} to the unique local solution $u\in C_T\mathcal C^{\alpha+1}$ found in Theorem \ref{thm: local fixed point for J} and in Theorem \ref{thm: fixed point for J for small u0} we obtain that the unique mild solution is not only continuous in time, but it is actually smoother, more precisely $u\in  C_T^{\eps'}\mathcal C^{\alpha'+1}$, provided that $u_0\in\mathcal C^{\alpha+1+2\eps+\nu}$  for some small $\nu>0$ and $\eps>0$ chosen as in Lemma \ref{lm:J1}.
\end{rem}

\begin{lemma}\label{lm:J2}
Let Assumptions A1, A2 and A3 hold and let us choose $\eps >0 $ according to Lemma \ref{lm:J1}. Then the operator $J:  C_T^{\eps}\mathcal C^{\alpha+1} \to  C_T^{\eps}\mathcal C^{\alpha+1} $ is continuous.
\end{lemma}
\begin{proof}
From Lemma \ref{lm:J1}, the fact that $\eps'>\eps$ and $\alpha'>\alpha$ and the embeddings $ C_T^{\eps'}\mathcal C^{\alpha'+1}\subset  C_T^{\eps}\mathcal C^{\alpha+1}\subset C_T\mathcal C^{\alpha+1}$ we have that $J:  C_T^{\eps}\mathcal C^{\alpha+1} \to  C_T^{\eps}\mathcal C^{\alpha+1} $. To show continuity we take $u,v \in  C_T^{\eps}\mathcal C^{\alpha+1}$ and bound the sup norm and the H\"older semi-norm of the difference  $J(u) - J(v)$.

 The sup norm of $J(u) - J(v)$ is bounded by Propositions \ref{pr: bound for Iu - Iv} (with $\rho=1$) together with  the fact that the embedding $ C_T^{\eps}\mathcal C^{\alpha+1}\subset C_T\mathcal C^{\alpha+1}$ is continuous. Then one has
\[
\sup_{0\leq t\leq T} \|J_t(u)-J_t(v)\|_{\alpha+1} \leq c(1+\|u\|_{\eps, \alpha+1}^2+\|v\|_{\eps, \alpha+1}^2)^{1/2} \|u-v\|_{\eps, \alpha+1}.
\]
The H\"older semi-norm of  $J(u) - J(v)$   is bounded by splitting the integral similarly to what was done in \eqref{eq:Jint}. One obtains
\begin{align*}
\|J_t(u)-&J_t(v)- J_s(u)+J_s(v)\|_{\alpha+1} \\
 \leq & \|\int_0^s (P_{t-s}- I) (P_{s-r} \left( G_r(u)-G_r(v) \right) ) \ud r\|_{\alpha+1}\\ 
&+ \|\int_s^t P_{t-r}  \left( G_r(u)-G_r(v) \right)  \ud r\|_{\alpha+1}.
\end{align*}
Then we proceed similarly as for the bounds of $M_2$ and $M_3$
in the proof of Lemma \ref{lm:J1}, but using   \eqref{eq:Gdiff} instead of \eqref{eq:G}, and with $\eps, \alpha$ instead of $\eps', \alpha'$, to obtain
\begin{align*}
\|J_t(u)-&J_t(v)- J_s(u)+J_s(v)\|_{\alpha+1} \\
\leq & (t-s)^\eps \left( s^{\frac{-\alpha+1+\beta}2} + (t-s)^{\frac{-\alpha+1+\beta-2\eps}2} \right) \times \\
 &\times c(1+\|u\|_{\eps, \alpha+1}+\|v\|_{\eps, \alpha+1})^{1/2} \|u-v\|_{\eps, \alpha+1}.
\end{align*}
Thus
\begin{align*}
\sup_{0\leq s< t\leq T} &\frac{\| J_t(u)-J_t(v)- J_s(u)+J_s(v)\|_{\alpha+1} }{(t-s)^\eps}\\
&\leq c T^{\frac{-\alpha+1+\beta-2\eps}2}(1+\|u\|_{\eps, \alpha+1}+\|v\|_{\eps, \alpha+1})^{1/2} \|u-v\|_{\eps, \alpha+1}
\end{align*}
and the proof is complete. 
\end{proof}

\begin{lemma}\label{lm:J3}
Let Assumptions A1, A2, A3 and A4 hold and let $\eps$ be chosen as in Lemma \ref{lm:J1}. Let $u_0\in \mathcal C^{\alpha+1+ 2\eps +\nu}$ for some small $\nu>0 $. Then the set 
\[
\Lambda:=\{u\in C_T^\eps\mathcal C^{\alpha+1} \text{ such that } u=\lambda J(u) \text{ for some } \lambda \in[0,1]\}
\]
is bounded in $ C_T^\eps\mathcal C^{\alpha+1} $.
\end{lemma}
\begin{proof}
Let $u^*\in \Lambda$, that is $ u^*=\lambda J(u^*) $ for some $\lambda \in[0,1]$. Applying Lemma \ref{lm:J1} and Proposition \ref{pr:priori}  we get
\begin{align*}
\|u^*\|_{\eps, \alpha+1} \leq & \|J(u^*)\|_{\eps, \alpha+1}\\
\leq & c\|u_0\|_{\alpha +1+2\eps+\nu}  + cT^{\frac{-\alpha'+1+\beta-2\eps}2} (1+\|u^*\|_{C_T \mathcal C^{\alpha+1}} )\\
\leq & c\|u_0\|_{\alpha +1+2\eps+\nu} + cT^{\frac{-\alpha'+1+\beta-2\eps}2}   (1+ K),
\end{align*}
where the constant on the right hand side is finite and  independent of $u^*$. 
\end{proof}

\begin{theorem}\label{thm:global}
Let Assumptions A1, A2, A3 and A4 hold and let $\eps>0$ be chosen according to Lemma \ref{lm:J1}. If $u_0 \in \mathcal C^{\alpha+1+2\eps+\nu}$ for some small $\nu>0$, then there exists a global mild  solution $u$ of \eqref{eq: PDE non lin Cauchy prb} in $ C_T^\eps\mathcal C^{\alpha+1} $ which is unique in $ C_T \mathcal C^{\alpha+1} $.
\end{theorem}
\begin{proof}
\emph{Existence.} By Lemma \ref{lm:J1} we have that 
\[
J:  C_T^\eps\mathcal C^{\alpha+1} \to  C_T^\eps\mathcal C^{\alpha+1}
\]
and by Lemma \ref{lm:J2} we know that $J$ is also continuous. Moreover using  Lemma \ref{lm:J1} again we have that the operator $J$ maps balls of $ C_T^\eps\mathcal C^{\alpha+1}$ into balls of $C_T^{\eps'}\mathcal C^{\alpha'+1}$ for some $\eps'>\eps$ and $\alpha'>\alpha$, which are pre-compact sets in  $ C_T^{\eps}\mathcal C^{\alpha+1}$. Thus $J$ is compact. We conclude that $J$ has a fixed point $u^*$ in $ C_T^{\eps}\mathcal C^{\alpha+1}$ by Schauder's fixed point theorem and by Lemma \ref{lm:J3}. The fixed point $u^*$ is a mild solution of \eqref{eq: PDE non lin Cauchy prb} in $ C_T^\eps\mathcal C^{\alpha+1} $.\\
\emph{Uniqueness.} Clearly $u^*\in   C_T \mathcal C^{\alpha+1} $. This solution is unique in the latter space by Remark \ref{rm:uniqueness1}.
\end{proof}

\section{Applications to stochastic analysis}\label{sc: applications}

In this section we illustrate an application of non-linear singular PDEs   to stochastic analysis, in particular to a class of non-linear backward stochastic differential equations (BSDEs) with distributional coefficients. The class of BSDEs that we consider here has not been studied previously in the  BSDEs literature.

The concept of a  BSDE was introduced in the early 90s by Pardoux and Peng \cite{pardoux-peng}. Since then, BSDEs have become a popular research field and the literature on this topic is now vast, see for example two recent books \cite{pardoux-rascanu14, zhang} and references therein. BSDEs own their success to the many applications they have in other areas of research. The main ones are their use in financial mathematics for pricing and hedging derivatives; their application to stochastic control theory  to find the optimal control and the optimal value function; and their use in showing existence and uniqueness of solutions to certain classes of non-linear PDEs by means of a  probabilistic representation of their solution (known as non-linear Feynman-Kac formula). 

The application that we are going to illustrate below fits in the latter two of these three topics. Indeed, the singular PDE studied above will allow us to define and solve a singular BSDE which is linked to the PDE by an extended Feynman-Kac formula. Moreover this class of  BSDEs arises also in stochastic control when looking at problems in  Economics where an agent wants to  maximise her exponential utility, see for example \cite[Chapter 20]{bjork} and \cite[Chapter 7]{zhang}. 
This latter class of BSDEs is known as quadratic BSDEs and  is linked to the special non-linearity $F(x)=x^2$. Note that in this section we   restrict to one space dimension.  This restriction and the choice of quadratic $F$ are done to avoid technicalities, but it should be a simple exercise to extend the argument below to a  general non-linear $F$ satisfying Assumption A1 and such that $F(0)=0$. The multidimensional case $(d>1)$ should also be possible to treat, much in the spirit of \cite{IssoglioJing16}.  Details of this are left to the interested reader and to future work. 

\vspace{10pt}
 
Let us start by writing the PDE \eqref{eq: PDE non lin Cauchy prb} in one-dimension and backward in time, which is the classical  form (Kolmogorov backward equation) when dealing with BSDEs:
\begin{equation}\label{eq: PDE for BSDE}
\left\{\begin{array}{ll}
\partial_t u(t,x)  + \partial _{xx} u(t,x) + (\partial_x u(t,x))^2  b(t,x)=0, & \text{ for } (t,x)\in [0,T]  \times\mathbb R \\
u(T,x)=\Phi (x),  & \text{ for } x\in \mathbb R .
\end{array}\right.
\end{equation}
We observe that (by abuse of notation) we used the same symbol $u$ as in the forward PDE and we denoted by $\Phi$ rather than $u_0$ the final condition. This is done to  be in line with classical BSDEs notation. The results of Section \ref{sc: solving the PDE} and in particular  Theorem \ref{thm: fixed point for J for small u0} apply to this PDE because the only  difference from \eqref{eq: PDE non lin Cauchy prb} is the time-change. Indeed  it is easy to check that  $F(x)= x^2$ satisfies Assumption A1 and moreover $F(0)=0$. 
\begin{rem} 
Since here we want to work in a given time-interval $[0,T]$ then we must ensure that the terminal condition $\Phi$ is small enough according to Theorem \ref{thm: fixed point for J for small u0}.
\end{rem}
Given a probability space $(\Omega, \mathcal F, \mathbb P)$ we consider a  BSDE of the form
\begin{equation}\label{eq: BSDE}
Y_r^{t,x} = \Phi(B_T^{t,x}) +\int_r^T b(s, B_s^{t,x}) (Z^{t,x}_s)^2 \mathrm ds - \int_r^T Z^{t,x}_s \mathrm dB^{t,x}_s,
\end{equation}
where $B:=(B^{t,x}_r)_{t\leq r\leq T}$ is a Brownian motion starting in $x$ at time $t$ and with quadratic variation $2r$ at time $r\geq t$. This latter non-standard quadratic variation is introduced to account for the fact that the generator of  Brownian motion is $\frac12 \partial_{xx}$ but the operator in the PDE \eqref{eq: PDE for BSDE} is $\partial_{xx}$.  The Brownian motion $B$ generates a filtration $\mathbb F:=(\mathcal F_r)_{t\leq r\leq T }$. 
It is known that if $b$ and $\Phi$ are smooth enough functions and satisfy some bounds (see e.g. \cite[Theorem 7.3.3]{zhang}) then the solution to the BSDE exists and it is unique. Note that a solution to \eqref{eq: BSDE}  is a \emph{couple} of adapted processes $(Y^{t,x},Z^{t,x})$ that satisfies \eqref{eq: BSDE}  and some other integrability  conditions (like the ones in the second bullet point of Definition \ref{def: virtual solution} below).
Moreover it is know that, in the classical case, the  BSDE and the PDE above are linked via the Feynman-Kac formula, namely  $ Y^{t,x}_r = u(r, B_r^{t,x}), \text{ and } Z^{t,x}_r = \partial_x u(r, B_r^{t,x})$.\footnote{One side of the Feynman-Kac formula can be easily checked, namely that the couple $( u(r, B_r^{t,x}), \partial_x u(r, B_r^{t,x})) $ is a solution of the BSDE. This is done by applying It\^o's formula to $ u(r, B_r^{t,x})$.} In particular for the initial time $t$ one gets the stochastic representation for the solution of the PDE \eqref{eq: PDE for BSDE} in terms of the solution of the BSDE \eqref{eq: BSDE}, namely 
\[
u(t,x) = Y_t^{t,x}.
\]

In the remaining of this section  we are going to use the results on the singular parabolic PDE to solve  the singular BSDE \eqref{eq: BSDE} when $b\in L_T^\infty \mathcal C^{\beta}$. 
One of the delicate points here is to give a meaning to the term $\int_r^T b(s, B_s) Z_s^2 \mathrm ds $, which we do by using the \emph{It\^o trick}. The {It\^o trick} has been used in the past to  treat other SDEs and BSDEs with distributional coefficients, see e.g.\  \cite{flandoli_et.al, IssoglioJing16}. This trick makes use of the following auxiliary PDE
 \begin{equation}\label{eq: auxiliary PDE for BSDE}
\left\{\begin{array}{ll}
\partial_t w(t,x)  +  \partial_{xx} w(t,x) = (\partial_x u(t,x))^2  b(t,x), & \text{ for } (t,x)\in [0,T]  \times\mathbb R \\
w(T,x)=0, & \text{ for } x\in \mathbb R ,
\end{array}\right.
\end{equation}
where the function $u$ appearing on the right-hand side is the solution to \eqref{eq: PDE for BSDE}. The mild form of this PDE is given  by
\begin{equation*}\label{eq: mild solution w}
w(t) = - \int_t^T P_{s-t} \left( ( \partial_x u(s))^2  b(s) \right) \mathrm ds.
\end{equation*}

Let us now do some \emph{heuristic} reasoning. If $b$ was smooth, then applying It\^o's formula to $w(r, B^{t,x}_r)$  would give 
\begin{align*}
\int_r^T \mathrm d  w(s, B_s^{t,x}) =& \int_r^T \partial_t w(s, B_s^{t,x}) \mathrm d s +  \int_r^T \partial_x w(s, B_s^{t,x}) \mathrm dB^{t,x}_s \\
&+ \frac12 \int_r^T \partial _{xx} w(s, B^{t,x}_s) 2\mathrm ds\\
=& \int_r^T\partial_x w(s, B^{t,x}_s) \mathrm dB^{t,x}_s + \int_r^T(\partial_x u(s,B^{t,x}_s))^2  b(s,B^{t,x}_s) \mathrm ds.
\end{align*} 
Moreover, if $b$ was smooth, then the classical theory on BSDEs ensures that  $Z_r = \partial_x u(r, B^{t,x}_r)$, so integrating the above equation one has
\begin{align*}
w(T, B_T)- w(r, B_r^{t,x}) &=  \int_r^T \partial_x w(s, B^{t,x}_s) \mathrm dB^{t,x}_s + \int_r^T  (Z^{t,x}_s)^2  b(s,B_s^{t,x}) \mathrm ds.
\end{align*} 
Thus we can express the singular term including $b$ in terms of quantities that are well defined and do not depend on $b$ explicitly, namely
\begin{equation}\label{eq: virtual term} 
 \int_r^T  (Z^{t,x}_s)^2  b(s,B_s^{t,x}) \mathrm ds = - w(r, B_r^{t,x}) - \int_r^T \partial_x w(s, B^{t,x}_s) \mathrm dB^{t,x}_s .
\end{equation}
We note that even in the singular case when $b\in L_T^\infty \mathcal C^\beta$ we have that all terms on the right hand side of \eqref{eq: virtual term} are well defined. Indeed using  the regularity of $u$, $b$ and their product (see \eqref{eq: Bony's estimates}) together with Lemma \ref{lm: continuity of I} one has that  $w\in C_T\mathcal C^{\alpha+1}$ and therefore  $w$ is differentiable  (in the classical sense) once in $x$, so $\partial_x w (s,x)$ is well defined.

The idea of the It\^o trick is to ``replace'' the singular integral term with the right-hand side of \eqref{eq: virtual term}, which is the motivation for the following definition. Note that we drop the superscript $\cdot^{t,x}$ for ease of notation. 
\begin{defin}\label{def: virtual solution}
A couple $(Y,Z)$ is called \emph{virtual solution} of \eqref{eq: BSDE} if
	\begin{itemize}
		\item $Y$ is continuous and $\mathbb F$-adapted and $Z$ is $\mathbb F$-progressively measurable; 
		\item  $ E \left[ \sup_{r\in[t,T]}|Y_r|^2 \right ] < \infty$ and $ E\left[ \int_t^T |Z_r|^2 \mathrm dr\right] < \infty$;
		\item for all $r\in[t,T]$, the couple satisfies the following  backward SDE 
		\begin{align}
			\label{eq: BSDE virtual}
			{Y}_r= & \ \Phi(B_T)- w(r, B_r) - \int^T_r ({Z}_s +  \partial_x w(s, B_s)  )\mathrm d  B_s
		\end{align}
		$\mathbb P$-almost surely.
	\end{itemize}
\end{defin}

We now observe that BSDE \eqref{eq: BSDE virtual} can be transformed into a classical BSDE by setting $\hat Y_r := Y_r+ w(r, B_r)$ and $\hat Z_r := Z_r + \partial_x w(r, B_r)$. One has that \eqref{eq: BSDE virtual} is equivalent to 
\begin{equation}\label{eq: BSDE virtual transformed}
\hat Y_r=   \Phi(B_T) - \int^T_r  \hat Z_s    \mathrm d  B_s,
\end{equation}
thus the $\hat Y$ component in \eqref{eq: BSDE virtual transformed} is given explicitly  by $\hat Y_r = \mathbb E\left [ \Phi (B_T) \vert \mathcal F_r \right ]$.
Moreover by the martingale representation theorem (see e.g.\ \cite[Theorem 2.5.2]{zhang}) there exists a unique predictable process  $\hat Z$ such that $ \hat Y_r =   \hat Y_t + \int^r_t  \hat Z_s    \mathrm d  B_s$ and so  $\hat Y_r=  \hat Y_T - \int^T_r  \hat Z_s    \mathrm d  B_s $.
Therefore given the transformation $w$, we can find explicitly the virtual solution of \eqref{eq: BSDE} by 
\begin{equation}\label{eq: sol BSDE}
 Y_r = \mathbb E\left [ \Phi (B_T) \vert \mathcal F_r \right ] -  w(r, B_r),
 \text{ and }
 Z_r = \hat Z_r -  \partial_x w(r, B_r).
\end{equation}
What we explained above can be summarised in the following theorem.
\begin{theorem}
If $b\in L_T^\infty\mathcal C^\beta$, then there exists a unique virtual solution $(Y,Z)$ of \eqref{eq: BSDE} given by \eqref{eq: sol BSDE}.
\end{theorem}

\begin{rem}
It is easy to check that the notion of virtual solution coincides with the classical solution when $b$ is smooth, because the heuristic argument explained above to motivate \eqref{eq: virtual term} is actually rigorous. Indeed this is the case if $b\in L_T^\infty\mathcal C^{\beta}$ is also a function smooth enough so that $u\in C^{1,2}$ and so that the BSDE can be solved with classical theorems (see e.g.\ \cite[Chapter 7]{zhang}).
\end{rem}

The notion of virtual solution for BSDEs has been previously  used in \cite{IssoglioJing16} for the linear case when $F(x) = x $. There the authors show existence and uniqueness of a virtual solution for the corresponding BSDE similarly as what has been done here but for a slightly different class of drifts that live in Triebel-Lizorkin spaces rather than  Besov spaces. Moreover for the linear case $F(x)=x$ it has been shown in \cite{issoglio_russo} that the virtual solution introduced in \cite{IssoglioJing16} indeed coincides with a solution to the BSDE defined directly (hence by giving a meaning to the singular term instead of replacing it with known terms via  the It\^o trick). This was achieved with the introduction of an integral operator $A$ to represent the singular integral. 

It will be objective of future research to investigate the existence of an integral operator $A$ related to the non-linear term $F(x)$ analogously to the integral operator introduced in \cite{issoglio_russo}, and give a meaning to the BSDE directly rather than via the It\^o trick as done here. 
 
\section*{Acknowledgment}
The author would like to thank the anonymous referee for providing useful comments and hints, in particular regarding Section \ref{sc:global}.


\end{document}